\documentclass[11pt]{amsart}

\usepackage{amsthm,amssymb,amsfonts,amsmath,enumerate,graphicx,color,tikz}

\DeclareMathOperator{\lk}{lk}
\DeclareMathOperator{\sd}{sd}
\DeclareMathOperator{\rk}{rk}
\DeclareMathOperator{\st}{st}
\DeclareMathOperator{\id}{id}
\DeclareMathOperator{\Sym}{Sym}
\DeclareMathOperator{\odd}{odd}

\def\C{{\mathcal C}}
\def\G{{\mathcal G}}
\def\phi{\varphi}

\newcommand\commentout[1]{}

\newtheorem{theorem}{Theorem}[section]
\newtheorem{corollary}[theorem]{Corollary}
\newtheorem{proposition}[theorem]{Proposition}
\newtheorem{lemma}[theorem]{Lemma}

\theoremstyle{remark}
\newtheorem{example}[theorem]{Example}
\newtheorem{remark}[theorem]{Remark}
\newtheorem{problem}{Problem}

\theoremstyle{definition}
\newtheorem{definition}[theorem]{Definition}

\begin{document}

\title{Simplicial moves on balanced complexes}

\author{Ivan Izmestiev}
\address{University of Fribourg Department of Mathematics\\ Chemin du Mus\'ee 23 \\ CH-1700 Fribourg P\'erolles \\ SWITZERLAND}
\email{ivan.izmestiev@unifr.ch}
\thanks{I.I. was partially supported by the European Research Council under the European Union's Seventh Framework Programme (FP7/2007-2013)/\allowbreak ERC Grant agreement no.~247029-SDModels.}

\author{Steven Klee}
\address{Seattle University Department of Mathematics \\ 901 12th Avenue \\
Seattle, WA 98122, USA}
\email{klees@seattleu.edu}

\author{Isabella Novik}
\address{University of Washington Department of Mathematics \\ Box 354350 \\
Seattle, WA 98195, USA}
\email{novik@math.washington.edu}
\thanks{I.N. was partially supported by the National Science Foundation grant DMS-1361423}

\date{\today}

\thanks{}

\begin{abstract}
We introduce a notion of cross-flips: local moves that transform a balanced (i.e., properly $(d+1)$-colored)  triangulation of a combinatorial $d$-manifold into another balanced triangulation. These moves form a natural analog of bistellar flips (also known as Pachner moves).
Specifically, we establish the following theorem: any two
balanced triangulations of a closed combinatorial $d$-manifold can be connected by a sequence of cross-flips. Along the way we prove that for every $m \geq d+2$ and any closed combinatorial $d$-manifold $M$, two $m$-colored triangulations of $M$ can be connected by a sequence of bistellar flips that preserve the vertex colorings.
\end{abstract}

\maketitle


\section{Introduction} 

Pachner \cite{Pachner-87, Pachner-91} introduced a finite set of moves called \emph{bistellar flips} or \emph{Pachner moves} that change the combinatorial structure of a simplicial complex without changing its homeomorphism type. Pachner proved the following surprising property: if $\Delta_1$ and $\Delta_2$ are two PL homeomorphic closed combinatorial manifolds, then they can be connected by a sequence of bistellar flips (see \cite{Pachner-87,Pachner-91} and also \cite[Theorem 5.9]{Lickorish-99}).

We cannot avoid mentioning the following spectacular applications of this result.
\begin{itemize}
\item The Turaev--Viro invariant \cite{Turaev-10,Turaev-Viro-92} associates a number to every triangulation of a $3$-manifold. In order to see that this is a manifold invariant, independent of the choice of a triangulation, one shows that it does not change under bistellar flips.
\item The celebrated $g$-theorem of Billera, Lee, McMullen, and Stanley \cite{Billera-lee-81,McMullen-71,Stanley-80} provides a complete characterization of possible face numbers of simplicial polytopes, and the $g$-conjecture posits that the same statement holds for the face numbers of all simplicial spheres. Although Stanley's original proof of the necessity of conditions of the $g$-theorem for simplicial polytopes relied on hard results from algebraic geometry, in the 1990s, McMullen \cite{McMullen-93,McMullen-96} found a more elementary proof using bistellar flips and results from convex geometry. 
\item While the $g$-conjecture for spheres is wide open at present, Swartz \cite{Swartz-14} recently proved it for any even-dimensional combinatorial sphere that can be obtained from the boundary complex of a simplicial polytope by using all but one particular bistellar flip \cite[Theorem~3.1]{Swartz-14}. 
\end{itemize}

Another important consequence of Pachner's result is that it allows one to search through the space of closed combinatorial manifolds. This approach is central to several software packages, for instance, BISTELLAR that Lutz started developing in mid-late 1990s \cite{BISTELLAR}. In the 20 years since then, this activity has led to several outstanding results in combinatorial topology including (a) many new examples of (vertex) minimal triangulations, (b) enumerating and cataloging vertex transitive triangulations of closed combinatorial manifolds of small dimension and with a small number of vertices (see \cite{manifold-page}), and even (c) characterizing all possible face numbers of twenty different $3$-manifolds \cite{Lutz-Sulanke-Swartz-09}. 

In this paper we discuss vertex-colored simplicial complexes.
A \emph{proper $m$-coloring} of a simplicial complex is a map from its vertex set to a set of cardinality $m$ (whose elements are called colors) such that adjacent vertices receive different colors. It is easy to see that every $d$-dimensional complex requires at least $d+1$ colors to be properly colored. One of our results is the following colored version of the Pachner theorem.

\begin{theorem}
\label{color-thm}
Let $\Delta$ and $\Gamma$ be PL homeomorphic closed combinatorial $d$-manifolds. Assume that $\Delta$ and $\Gamma$ are properly $m$-colored, $m \ge d+2$. Then there is a sequence of bistellar flips that transforms $\Delta$ into $\Gamma$ such that each intermediate complex is properly $m$-colored and the flips preserve the vertex colors.
\end{theorem}
\noindent If a bistellar flip introduces a new vertex to the complex, then this vertex can be colored in any of $m$ colors that is different from the colors of its neighbors.

Our primary focus in this paper will be on balanced simplicial complexes.  A $d$-dimensional simplicial complex is called balanced if it admits a proper vertex coloring with a set of $d+1$ colors.  These complexes were introduced by Stanley \cite{Stanley-79} where they were called completely balanced complexes.
Balanced complexes form a fascinating class of objects that arise often in combinatorics, algebra, and topology. For instance, the barycentric subdivision of any regular CW complex is balanced (and PL homeomorphic to the original complex); therefore, every triangulable space has a balanced triangulation. Coxeter complexes and Tits buildings form another large family of balanced complexes.

Bistellar flips may destroy the balanced property, so that Theorem \ref{color-thm} does not hold for $m=d+1$.
The main objective of this paper is to develop a balanced substitute for bistellar flips: local moves that preserve balancedness while still allowing one to connect any pair of PL homeomorphic balanced combinatorial manifolds.

Roughly speaking (we defer all precise definitions until the following sections), a bistellar flip exchanges a $d$-ball in the boundary of the $(d+1)$-simplex with its complement. In the balanced setting, the cross-polytope often serves as a substitute for the simplex.  For example, just as the boundary of the $(d+1)$-simplex is a minimal triangulation of the $d$-sphere, the boundary of the $(d+1)$-dimensional cross-polytope is the minimal balanced triangulation of the $d$-sphere; for deeper parallels between the two see \cite{Klee-Novik-14}. Thus it is natural to define a \emph{cross-flip} as an operation that exchanges a $d$-ball in the boundary of the $(d+1)$-cross-polytope with its complement. If this move is applied to a balanced complex, the resulting complex is also balanced.

With this definition in hand, our main result is the following.
\begin{theorem} \label{main-thm}
Let $\Delta$ and $\Gamma$ be balanced simplicial complexes.
Assume further that $\Delta$ and $\Gamma$ are closed combinatorial manifolds.
Then $\Delta$ and $\Gamma$ are  PL homeomorphic if and only if there is a
sequence of cross-flips that transforms $\Delta$ into $\Gamma$.
\end{theorem}

\noindent In particular, every balanced combinatorial sphere can be obtained from
the boundary of a cross-polytope by a sequence of cross-flips.

A cross-flip not only preserves the balanced property, but also induces a 
coloring of the transformed complex from a coloring of the complex to which the 
cross-flip was applied. The corresponding strengthening of Theorem \ref{main-thm} is 
still true: there is a sequence of cross-flips that transforms 
$\Delta$ into $\Gamma$ and a given coloring of $\Delta$ into a given coloring of $\Gamma$.

We expect that Theorem \ref{main-thm} will have many applications in the balanced
setting just as Pachner's theorem has in the non-balanced one. For instance, we
hope that this theorem will lead to the proof of the $g$-conjecture for
\emph{all} balanced combinatorial spheres. 

The rest of the paper is structured as follows. In Section 2 we discuss
basics of simplicial complexes, recall a few theorems, and introduce several key
definitions including the definition of a cross-flip. Section 3 can be considered
as a warm-up section: there, after developing some tools, we prove Theorems \ref{color-thm} and
\ref{main-thm} for the case of combinatorial spheres. The proof of the general
case of Theorems \ref{color-thm} and \ref{main-thm} is much more technical, and requires more
definitions and preparation; these are discussed in Section 4. Finally, in
Section 5, we prove Theorems \ref{color-thm} and \ref{main-thm} in their full generality. We close in Section 6
with some remarks and open questions.

\section{Background and definitions -- part I}
In this section we discuss basics of simplicial complexes, bistellar flips and
shellability; we also introduce our main definition --- the notion of a
cross-flip.

\subsection{Simplicial complexes and combinatorial manifolds}
A \textit{simplicial complex} $\Delta$ on a (finite) vertex set $V = V(\Delta)$
is a collection of subsets $F \subseteq V(\Delta)$ called \textit{faces} with the
property that if $F \in \Delta$ and $G \subseteq F$, then $G \in \Delta$.  The
\textit{dimension} of a face is $\dim(F) = |F|-1$ and the dimension of $\Delta$
is $\dim(\Delta) = \max\{\dim(F) \mid F \in \Delta\}$.  A \textit{facet} of
$\Delta$ is a maximal face under inclusion, and we say that $\Delta$ is
\textit{pure} if all of its facets have the same dimension.
If $\Gamma$ and $\Delta$ are simplicial complexes on
disjoint vertex sets, their \textit{join} is the simplicial complex
\[
\Gamma\ast\Delta:= \{F \cup G \mid F \in \Delta \text{ and } G \in \Gamma\}.
\]

If $F$ is a finite set, we write $\overline{F}:= \{G \subseteq F\}$ to denote the simplex
on $F$ and $\partial \overline{F}:=\{G \subsetneq F\}$ to denote the boundary
complex of the simplex on $F$.  When $F = \{a\}$ consists of a single vertex, we write $\overline{a}$ to denote the vertex $a$, viewed as a $0$-dimensional simplicial complex. If the vertex set of the simplex is not
important, we use $\sigma^d$ and $\partial \sigma^d$ to denote a
$d$-simplex and its boundary complex, respectively.

Let $\Delta$ be a simplicial complex and $F \in \Delta$ be a face. The (closed) \emph{star} of
$F$ in $\Delta$ and the \emph{link} of $F$ in $\Delta$ both describe the local
structure of $\Delta$ around $F$:
\[
\st_{\Delta}(F):= \{G \in \Delta \mid F \cup G \in \Delta\}, \quad \lk_{\Delta}(F):= \{G \in \st_\Delta(F) \mid F \cap G = \emptyset\}.
\]
Note that $\st_{\Delta}(F) = \overline{F} \ast \lk_{\Delta}(F)$.
The \emph{deletion} of $F$ from $\Delta$ is defined as
\[
\Delta\setminus F=\{G \in \Delta \mid F\not\subseteq G\}.
\]

A \emph{combinatorial $d$-sphere} (respectively, a {\emph{combinatorial $d$-ball}}) is a simplicial complex PL homeomorphic to $\partial\sigma^{d+1}$ (respectively, $\sigma^d$). A closed \textit{combinatorial $d$-manifold} is a connected simplicial complex with the property that the link of each vertex is a combinatorial $(d-1)$-sphere.

\subsection{Stellar moves}
Let $\Delta$ be a simplicial complex and $F$ a face of $\Delta$.  The
\textit{stellar subdivision} of $\Delta$ at $F$ (also known as the
\textit{starring} at $F$) consists of (i) removing $F$
and all faces containing it, (ii) introducing a new vertex $a$ to
$V(\Delta)$, and (iii) adding new faces in $\overline{a} \ast \partial \overline{F} \ast
\lk_{\Delta}(F)$ to $\Delta$: 
\[
\sd_F(\Delta) := (\Delta \setminus F) \cup \left(\overline{a} \ast \partial
\overline{F} \ast \lk_{\Delta}(F)\right).
\]

Complexes $\Delta$ and $\Gamma$ are called \textit{stellarly equivalent} if one can be obtained from the other through a sequence of stellar subdivisions and their inverses (stellar welds).
The following is a classical result in combinatorial topology, a proof can be found in \cite[Theorem 4.5]{Lickorish-99}.

\begin{theorem}[Alexander \cite{Alexander-30}]
\label{thm:Alexander}
Two simplicial complexes are PL homeomorphic if and only if they are stellarly equivalent.
\end{theorem}

Alexander and, independently, Newman \cite{Newman-31} improved this by showing that subdivisions and welds on edges suffice.
Recently, Lutz and Nevo \cite{Lutz-Nevo-14} proved
a result of a similar flavor showing that PL homeomorphic flag manifolds can be
connected by a sequence of edge subdivisions and edge welds that preserve flagness.

It is an open problem whether any two PL homeomorphic complexes can be related by a sequence of stellar subdivisions and stellar welds so that all subdivisions appear before all welds. In other words, it is not known whether any two PL homeomorphic complexes admit a common refinement through stellar subdivisions.

\subsection{Bistellar moves}
Let $\Delta$ be a simplicial complex, and assume that
\begin{equation}
\label{eqn:BiMove1}
A \in \Delta, \quad B \notin\Delta, \quad \text{and} \quad \lk_\Delta (A) = \partial\overline{B},
\end{equation}
so that $\st_{\Delta}(A) = \overline{A}*\partial\overline{B}$.  Then the process of removing $\overline{A} \ast \partial\overline{B}$ and replacing it with $\partial \overline{A} \ast \overline{B}$ is called a bistellar move or a \textit{bistellar flip}:
\[
\Delta \to \Delta \setminus (\overline{A} \ast \partial\overline{B}) \cup (\partial \overline{A} \ast \overline{B}).
\]
The inverse of a bistellar flip is itself a bistellar flip ($A$ and $B$ exchange roles). Two complexes are called \textit{bistellar equivalent} if one can be obtained from the other through a sequence of bistellar flips.

For closed combinatorial manifolds the assumption \eqref{eqn:BiMove1} is equivalent to
\begin{equation}
\label{eqn:BiMove2}
A \cap B = \emptyset, \quad |A \cup B| = d+2 \quad \text{and} \quad \Delta_{A \cup B} = \overline{A}*\partial\overline{B}.
\end{equation}
(Here $\Delta_W:=\{F\in \Delta \mid F\subseteq W\}$ denotes the subcomplex of $\Delta$ induced by vertices in $W$.)
Indeed, \eqref{eqn:BiMove1} implies \eqref{eqn:BiMove2} for any pure $d$-dimensional simplicial complex. Conversely, assume that $\Delta$ is a closed combinatorial manifold, and $A$ and $B$ satisfy \eqref{eqn:BiMove2}. Then $\partial\overline{B} \subseteq \lk_{\Delta}(A)$ because $\overline{A}*\partial\overline{B} \subseteq \Delta$. And since $\lk_{\Delta}(A)$ and $\partial\overline{B}$ are PL spheres of the same dimension, it must be the case that $\lk_{\Delta}(A) = \partial\overline{B}$.

The bistellar flip with $|B|=1$ (so that $\partial\overline{B} = \{\emptyset\}$) is a stellar subdivision at a facet. More generally, every bistellar flip can be represented as a composition of a stellar subdivision (done first) and a stellar weld. Thus bistellar equivalence implies stellar equivalence. On the other hand, Pachner has shown that for closed combinatorial manifolds bistellar equivalence is as strong as stellar equivalence.

\begin{theorem}[Pachner \cite{Pachner-91}]
\label{homeomorphic-bistellar-equivalent}
Two closed combinatorial manifolds are PL homeomorphic if and only if they are
bistellar equivalent.
\end{theorem}
A proof of Theorem 2.2  can also be found in Lickorish's paper \cite[Theorem 5.9]{Lickorish-99}. As Lickorish points out, some of these ideas go back to Newman \cite{Newman-26}.

An important difference between stellar and bistellar moves is that there are finitely many combinatorially distinct bistellar moves but infinitely many stellar ones, even if we restrict our scope to combinatorial manifolds. Indeed, the link of an edge is a $(d-2)$-sphere;  and since for $d \ge 3$ there are infinitely many combinatorial types of $(d-2)$-spheres, there are infinitely many combinatorially distinct edge subdivisions. This makes bistellar flips more useful for applications, for instance those mentioned in the introduction.

\subsection{Shellings and shellability}  \label{subsect:shellability1}

A pure $d$-dimensional simplicial complex $\Delta$ is \textit{shellable} if its
facets can be ordered $F_0, F_1, \ldots, F_t$ in such a way that, for each $j >
0$, $\overline{F}_j \cap \bigcup_{i < j} \overline{F}_i$ is pure of dimension
$d-1$.  Equivalently, for each $j\geq 0$ there exists a face $r(F_j) \subseteq
F_j$ such that $\overline{F}_j \setminus \bigcup_{i < j} \overline{F}_i =
[r(F_j),F_j]$.  Here and henceforth, $[r(F_j),F_j] = \{G \mid r(F_j) \subseteq G
\subseteq F_j\}$ denotes the interval from $r(F_j)$ to $F_j$ in the poset of all
faces, ordered by inclusion.  The faces $r(F_j)$ are often called the
\textit{restriction faces} of the shelling.   Also, if $|F|=d+1$ has the
property that $F \notin \Delta$ but $\overline{F}$ intersects $\Delta$ along a
pure $(d-1)$-dimensional subcomplex of $\partial \overline{F}$, then the process
of adding $\overline{F}$ to $\Delta$ is called an \textit{elementary shelling}. 
In other words, a $d$-dimensional simplicial complex is shellable if it can be 
obtained from the $d$-simplex through a sequence of elementary shellings.

Bruggesser and Mani \cite{Bruggesser-Mani} showed that the boundary complex of a
simplicial polytope is shellable; and moreover, that a shelling can be chosen to
have the set of facets containing any given face as its initial segment.  Those and the
following results on shellability will be useful for our proofs.

\begin{theorem}[{Provan--Billera, \cite[Corollary 2.9]{Provan-Billera-80}}] 
\label{subdiv-shellability}
If $\Omega$ is a shellable complex and $F$ is a face of $\Omega$, then
$\sd_F(\Omega)$ is also shellable.
\end{theorem}

Recall that a subcomplex $\Delta$ of $\Omega$ is called induced (or full) if every face $F$ of $\Omega$ all of whose vertices are in $\Delta$ is also a face of $\Delta$.

\begin{theorem}[{Pachner, \cite[Theorem 5.8]{Pachner-91}}]
\label{shellable-ball}
For every combinatorial $d$-sphere $\Delta$, there exists a shellable
$(d+1)$-ball $\Omega$ whose boundary is $\Delta$.
Moreover, $\Omega$ can be chosen so that $\Delta$ is an induced subcomplex of
$\Omega$.
\end{theorem}

We will also make use of the following well-known connection between elementary shellings and bistellar flips.
\begin{lemma}
\label{lem:ShellToFlip}
If $\Omega_1$ and $\Omega_2$ are combinatorial $d$-balls, and $\Omega_2$ is obtained from $\Omega_1$ by an elementary shelling  operation, then the boundary complex of $\Omega_2$ is obtained from the boundary complex of $\Omega_1$ by a bistellar flip.
\end{lemma}

Shellability also appears in the definition of cross-flips, which will be given in the next subsection.

\subsection{Balanced simplicial complexes and cross-flips}
We say that a simplicial complex $\Delta$ is \textit{properly $m$-colorable} if
there exists a map $\kappa: V(\Delta) \rightarrow \{0,1,\ldots,m-1\}$ such that
$\kappa(u) \neq \kappa(v)$ for each edge $\{u,v\} \in \Delta$.  Following Stanley
\cite{Stanley-79}, we say a $d$-dimensional simplicial complex is
\textit{balanced} if it is properly $(d+1)$-colorable.
The graph of the $d$-simplex cannot be properly colored with fewer than $d+1$ colors
because its graph is the complete graph on $d+1$ vertices.  Therefore, balanced
complexes are those complexes that can be colored with a minimum number of colors.

One example of a balanced $d$-sphere is the boundary complex of the
$(d+1)$-dimensional cross-polytope, which we denote as $\C_d$. Specifically, $V(\C_{d}) = \{x_0, x_1,
\ldots, x_d\} \cup \{y_0, y_1, \ldots, y_d\}$ and the faces of $\C_d$ are all
possible subsets $F \subset V(\C_{d})$ with the property that $|F \cap
\{x_i,y_i\}| \leq 1$ for all $i$. In particular, defining
$\kappa(x_i)=\kappa(y_i)=i$ for all $0\leq i \leq d$ provides a proper
$(d+1)$-coloring of $\C_{d}$. The complex $\C_{d}$ has the minimal
number of faces in each dimension among all balanced $d$-spheres.

Stellar subdivisions destroy the balanced property. A balanced stellar subdivision can be defined by removing a face $F$ and all faces containing it, inserting the Schlegel diagram of a cross-polytope of dimension $\dim F + 1$, and joining it with the link of $F$. The first row in Figure \ref{fig:CrossFlips2} shows the balanced subdivisions of a triangle and of an edge in a $2$-dimensional manifold.
Balanced stellar subdivisions were introduced by Fisk in \cite[Section II]{Fisk-77a} and \cite[Section VI.3]{Fisk-77b}, and rediscovered by Izmestiev and Joswig in \cite{Izmestiev-Joswig-03}.

It is natural to ask if a balanced analog of Theorem \ref{thm:Alexander} holds: \emph{can any two PL homeomorphic balanced simplicial complexes be connected by a sequence of balanced stellar subdivisions and their inverses?} 
Very recently, Murai and Suzuki \cite{Murai-Suzuki} showed that the answer is negative already in dimension~$2$.
In the non-balanced setting, Theorem \ref{homeomorphic-bistellar-equivalent} about bistellar moves is proved with the help of  Theorem \ref{thm:Alexander} about stellar moves. This makes it more surprising that there is a balanced analog of Theorem~\ref{homeomorphic-bistellar-equivalent}.

In the non-balanced setting a bistellar flip exchanges a $d$-ball in $\partial \sigma^{d+1}$ with its complement.  Several recent results of Klee and Novik \cite{Klee-Novik-14} and Juhnke-Kubitzke and Murai \cite{Juhnke-Kubitzke-Murai-15} have shown that substituting the boundary of a cross-polytope in place of the boundary of a simplex leads to balanced analogues of many classical results, such as the Lower Bound Theorem or the Generalized Lower Bound Conjecture. We use this as a motivation to define a \emph{cross-flip}.

A proper pure $d$-dimensional subcomplex $D \subset \C_d$ is called \emph{co-shellable}, if its complement $\C_d \setminus D$ is shellable (here $\C_{d} \setminus D$ denotes the pure subcomplex of $\C_{d}$ whose facets are precisely the facets of $\C_{d}$ not belonging to $D$). Note that $D\subset \C_d$ is shellable and co-shellable if and only if a shelling of $D$ extends to a shelling of $\C_d$. For instance, if $D$ is both shellable and co-shellable, then a shelling order of $D$ followed by the reverse of a shelling order of $\C_d \setminus D$ provides a shelling of $\C_d$ that extends that of $D$. Since $\C_d$ is not extendably shellable for $d \ge 11$ \cite{Hall-04}, the shellability of $D$ does not imply its co-shellability.

\begin{definition} \label{cross-flip-def}
Let $\Delta$ be a balanced combinatorial $d$-manifold and suppose $D \subset \Delta$ is an induced subcomplex of $\Delta$ that is isomorphic to a shellable and co-shellable subcomplex of $\C_d$.
We say that the process of replacing $D$ with $\C_{d} \setminus D$ is a \textit{cross-flip} on $\Delta$.
\end{definition}
Clearly, the inverse of a cross-flip is also a cross-flip.
The condition for $D$ being an induced subcomplex of $\Delta$ guarantees that replacing $D$ with its complement produces a simplicial complex. Since shellable simplicial balls are combinatorial, each cross-flip replaces a combinatorial $d$-ball $D$ with another combinatorial $d$-ball $D'$ such that $\partial D = \partial D'$. Thus cross-flips preserve the PL homeomorphism type of a complex.

\begin{figure}[ht]
\begin{center}
\includegraphics[width=.8\textwidth]{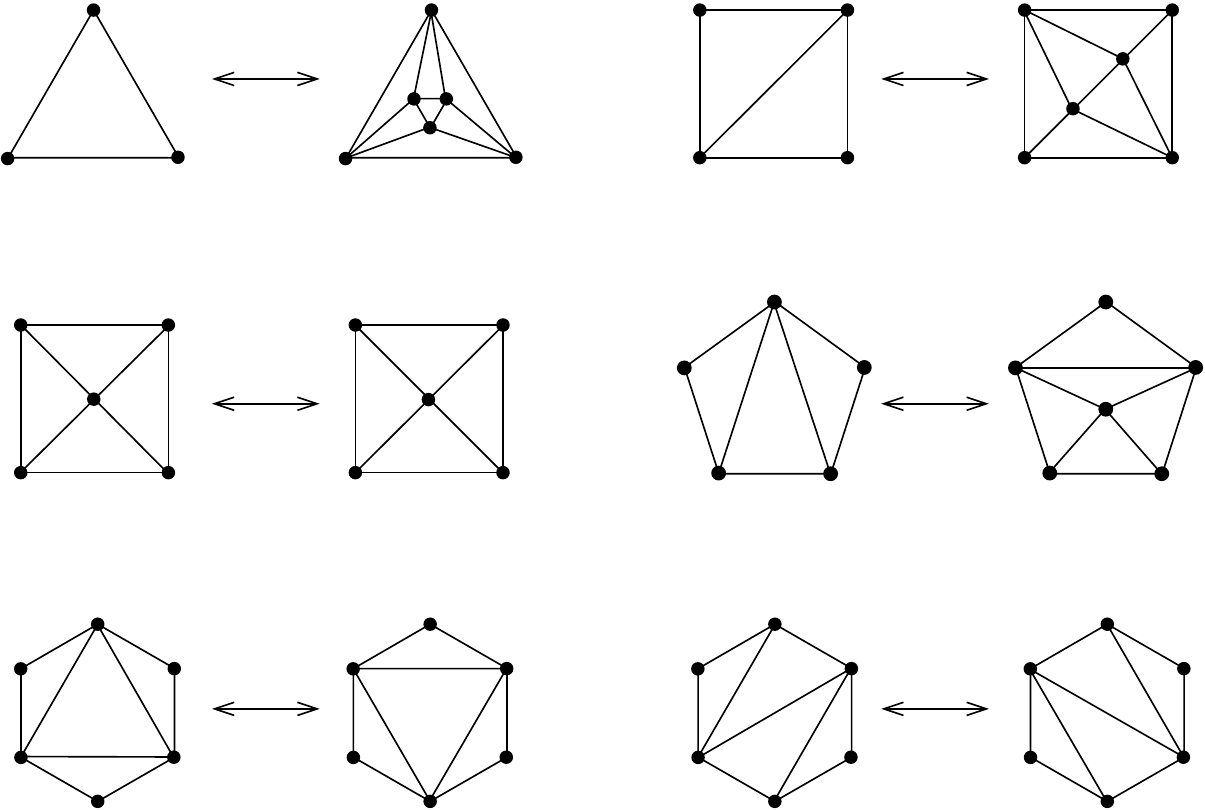}
\end{center}
\caption{Cross-flips in dimension $2$.}
\label{fig:CrossFlips2}
\end{figure}

\begin{example}
\label{exl:CrossFlips2}
Figure \ref{fig:CrossFlips2} shows all cross-flips on $2$-manifolds. Unlike the non-balanced situation, there is a trivial move, which we listed for completeness. In fact, the moves from the first row (balanced stellar subdivisions and welds) together with the pentagon move from the second row suffice to connect any two triangulations of a surface, see Remark \ref{rem:3Suffice}. 
 More surprisingly, Murai and Suzuki \cite{Murai-Suzuki} very recently proved that only three out of these six moves suffice, namely, balanced edge subdivisions and welds, and pentagon contractions (i.e., the pentagon move from the second row of Figure \ref{fig:CrossFlips2} that replaces 5 facets with 3 facets).

\end{example}

\section{A warm up: balanced spheres}
In this section, after developing some tools, we prove our main results --- Theorem \ref{color-thm} and Theorem \ref{main-thm} --- in the special case of spheres.

\subsection{Extending colorings through stellar subdivisions}
A \emph{relative simplicial complex} is a pair $(L, K)$ where $K$ is a subcomplex of $L$. A \emph{facet} of $(L, K)$ is a face $\sigma \in L \setminus K$ that is not properly contained in any other faces of $L \setminus K$. The \emph{dimension} of a relative simplicial complex is the maximum dimension of its facets. Hence $\dim (L, K) \le \dim L$, and this inequality is strict if and only if all top-dimensional faces of $L$ belong to $K$.

The following result provides a strengthening of \cite[Lemma 57]{Fisk-77b} and \cite[Theorem 3]{Izmestiev-05} and seems to be interesting in its own right.

\begin{theorem}  \label{ext-coloring}
Let $(L, K)$ be a $d$-dimensional relative simplicial complex.
Assume $\kappa : V(K) \to \{0,1,\ldots, m-1\}$ is a proper $m$-coloring of $K$. Then
there is a stellar subdivision $L'$ of $L$ such that
\begin{enumerate}
\item $K$ is a subcomplex of $L'$ (that is, no face of $K$ was subdivided), and
\item the coloring $\kappa$ extends to a proper coloring $\kappa' : V(L') \to
\{0,1,\ldots, \max\{m-1,d\}\}$ such that all vertices not in $K$ receive colors in
$\{0, 1, \ldots, d\}$.
\end{enumerate}
\end{theorem}

Note that the number $m$ might be larger than $\dim K + 1$. Also, the coloring $\kappa$ is assumed to be proper only with respect to $K$: there might be an edge in $L \setminus K$ whose endpoints belong to $K$ and are colored by $\kappa$ in the same color. Naturally, such an edge has to be subdivided in the process of transforming $L$ into $L'$.

\begin{proof}
We may assume without loss of generality that there are no edges in $L \setminus K$ whose endpoints are vertices in $K$ that receive the same color under $\kappa$. If any such edges do exist, we begin by stellarly subdividing them.

Define an initial coloring $\kappa'$ on $L$ by setting $\kappa'(u) = 0$ for each $u \in V(L) \setminus V(K)$ and $\kappa'(v) = \kappa(v)$ for all $v \in V(K)$. If the resulting coloring is proper, we are done. If not, then we proceed by performing special stellar subdivisions on $L$ and extending the coloring $\kappa'$ to the newly created vertices in a special way.

Call a face $F \in L$ \emph{dull}, if $\kappa(v) < \dim(F)$ for each $v \in F$. Note that every dull simplex is improperly colored, but not every improperly colored simplex is dull. However, any $\kappa'$-improperly colored edge of $L$ is dull. Indeed, by our initial assumption that no improperly colored edges in $L$ have both vertices in $K$, any improperly colored edge must have both of its vertices receive color $0$.

As a basic step, perform a stellar subdivision at an inclusion-maximal dull face
$F$, and color the newly introduced vertex $a$ in the color $\dim F$. Since there are no dull faces in $K$, we have $\dim F \le d$. We claim
that
\begin{enumerate}
\item[(i)] no faces created by a basic starring are dull, and
\item[(ii)] all edges created by a basic starring are properly colored.
\end{enumerate}
Assuming this claim, it follows that each basic stellar subdivision decreases the number of
dull faces and does not create any new improperly colored edges. Therefore we may repeat this until there are no more dull simplices. The resulting complex is properly
colored, since any improperly colored edge would be dull.  Further, we never subdivide any face of $K$, since these faces are colored properly.

Therefore, to complete the proof, it suffices to verify parts (i) and (ii) of the above claim. For property (i), note that after a basic starring at $F$,  every new face has the form
\[
\{a\} \cup F' \cup G, \quad \mbox{where } F' \subsetneq F \mbox{ and } G\in
\lk(F).
\]
If $G$ is empty, then such a new face is not dull because its dimension is at most $\dim F$ and it contains a vertex, $a$, of color $\dim F$. Thus assume that $\dim G = \ell \ge 0$, that $\dim F' = k' < k$ where $k=\dim F$, and that $\{a\} \cup F' \cup G$ is dull, that is, the colors of the vertices of this face are in $\{0, 1, \ldots, k'+\ell+1\}$. Then the face $F \cup G$ (before this basic starring) was dull: we have $\dim (F \cup G) = k+\ell+1 \ge k'+\ell+2$, and all the colors used are contained in $\{0, 1, \ldots, k+\ell+1\}$. But if $F\cup G$ was dull, then $F$ was not an inclusion-maximal dull face, which is a contradiction. 

Finally, to verify (ii), observe that the edges joining $a$ with vertices of $F$ are properly colored because $F$ is dull while the color of $a$ is $\dim F$. Also, every edge $\{a, b\}$ with $b$ in $\lk(F)$ is properly colored, since otherwise $\{b\} \cup F$ is dull, and hence $F$ is not an inclusion-maximal dull face. The result follows.
\end{proof}

The following result is immediate based on Theorem \ref{ext-coloring}.
\begin{corollary} \label{balanced->balanced}
Let $K$ be a subcomplex of $L$, and assume further that $K$ is balanced. Then there exists a stellar subdivision $L'$ of $L$ such that (i) $K$ is a subcomplex of $L'$, (ii) $L'$ is balanced, and (iii) the balanced coloring of $L'$ extends that of $K$.
\end{corollary}
\begin{proof}
Assume that $K$ is colored in $m = \dim K + 1$ colors. It suffices to consider two cases.

If $d = \dim(L, K) < \dim K = m-1$, then there exists a subdivision $L'$ colored in $m$ colors. At the same time $\dim L' = \dim L = \dim K = m$, so that $L'$ is balanced.

If $d = \dim(L, K) \ge \dim K = m-1$, then there exists a subdivision $L'$ colored in $d+1$ colors. At the same time $\dim L' = \dim(L', K) = d$, so that $L'$ is balanced.
\end{proof}

\begin{remark}
As suggested in \cite[Section 2.13]{Gromov-10}, Theorem \ref{ext-coloring} can be interpreted topologically. Assume that the coloring of $K$ uses $m \geq d+1$ colors. Let $X \supset K$ be the pure $(m-1)$-dimensional simplicial complex with the same vertex set as $K$, where the top-dimensional faces of $X$ are all $m$-subsets of $V(K)$ formed by points of different colors. The complex $X$ is $(m-2)$-connected --- a fact that is widely used in topological combinatorics \cite{Bjorner-95}. Since $\dim(L,K) = d\leq m-1$, the inclusion $|K| \subset |X|$ can be extended to a map $|L| \to |X|$.
In fact, Theorem \ref{ext-coloring} tells more: there is an extension by a non-degenerate piecewise-linear map.

If $m < d+1$, then add to $V(K)$ one vertex for each missing color and apply the same argument.
\end{remark}

\subsection{Connecting colored spheres}
We are now ready to provide the proof of Theorem \ref{color-thm} for the case of
spheres.

\begin{theorem}
\label{thm:SphereSimplex}
Let $m \ge d+2$. Then for every properly $m$-colored combinatorial $d$-sphere $\Delta$ there is a sequence of bistellar flips that transforms $\Delta$ into the boundary of a $(d+1)$-simplex such that each intermediate complex is properly $m$-colored and the flips preserve the vertex colors.
\end{theorem}
\begin{proof}
By Theorem \ref{shellable-ball}, there exists a shellable $(d+1)$-ball $\Omega$ whose boundary is $\Delta$. Applying Theorem \ref{ext-coloring} to the pair $(\Omega, \Delta)$ produces a properly $m$-colored $(d+1)$-ball $\Omega'$ such that $\partial\Omega'=\Delta$ and the coloring of $\Omega'$ extends that of $\Delta$. (Note that $\dim(\Omega,\Delta) = d+1$, and  $\max\{m-1, d+1\} = m-1$ because $m \geq d+2$.)

Since $\Omega'$ is a stellar subdivision of a shellable ball, it is shellable by Theorem \ref{subdiv-shellability}. By Lemma \ref{lem:ShellToFlip} every shelling order on $\Omega'$ defines a sequence of bistellar flips that transforms the boundary of a $(d+1)$-simplex to the boundary of $\Omega'$. Since $\Omega'$ is properly $m$-colored, the intermediate complexes are $m$-colored and the flips preserve the colors of the vertices.
\end{proof}

\begin{corollary}
Let $m \geq d+2$ and suppose $\Delta$ and $\Gamma$ are properly $m$-colored combinatorial $d$-spheres.  Then there is a sequence of bistellar flips that transforms $\Delta$ into $\Gamma$ and a given coloring of $\Delta$ into a given coloring of $\Gamma$ in such a way that each intermediate complex is properly $m$-colored and the flips preserve the colors of the vertices.
\end{corollary}
\begin{proof}
By Theorem \ref{thm:SphereSimplex}, $\Delta$ and $\Gamma$ can each be transformed to a copy of the boundary of the $(d+1)$-simplex whose vertices are properly $m$ colored (though these two copies may not be colored by the same set of $d+2 \leq m$ colors). Therefore it suffices to show that any two proper $m$-colorings of $\partial \sigma^{d+1}$ are related through a sequence bistellar flips. More specifically, it suffices to transform any proper $m$-coloring to the one using the colors $\{0, 1, \ldots, d\}$. This is done as follows. Assume a coloring uses a color $k > d$ and does not use a color $\ell \le d$. Subdivide the facet opposite to the vertex of color $k$, coloring the newly introduced vertex with $\ell$. Then remove the vertex of color $k$ with a stellar weld to get a new coloring of $\partial \sigma^{d+1}$ that uses $\ell$ instead of $k$. Repeating this procedure reduces the set of colors to $\{0, 1, \ldots, d\}$.
\end{proof}

When $m = d+1$ the argument in the proof of Theorem \ref{thm:SphereSimplex} breaks down.  A $(d+1)$-coloring of $\Delta$ cannot be extended to a $(d+1)$-coloring of a ball $\Omega'$ because a proper coloring of $\Omega'$ requires at least $d+2$ colors. The best that we can have is a $(d+2)$-colored ball $\Omega'$. In the next section we show how to transform $\Omega'$ into a $(d+1)$-colored cross-polytopal complex.

\subsection{The diamond operation: turning simplices into cross-polytopes}
\label{sec:Diamond}
The following lemma provides a systematic way of embedding the boundary complex of
the $(d+1)$-dimensional cross-polytope in the boundary of a $(d+1)$-simplex and
is one of the key ideas needed for our proofs.

\begin{lemma} \label{lem:StarToCross}
Let $F_1 \subset F_2 \subset \cdots \subset F_d$ be a flag of faces in $\partial \sigma^{d+1}$ such that $\dim(F_i) = i$ for all $i$.  The successive stellar subdivision of these faces produces the boundary complex of a cross-polytope: 
\[
\sd_{F_1} \circ \sd_{F_2} \circ \cdots \circ \sd_{F_d}(\partial \sigma^{d+1})
\cong \C_{d}.
\]
\end{lemma}
\begin{proof}
Let us prove a more general statement that for any $n$ and $1\leq k \leq d< n$,
\[
\sd_{F_k} \circ \sd_{F_{k+1}} \circ \cdots \circ \sd_{F_d}(\partial
\sigma^n) \cong \partial \overline{F_k} * \C_{d-k-1} * \partial \sigma^{n-d}
\]
(with $\C_{-1} = \{\emptyset\}$).

Note that if $F \in \partial \sigma^n$, then 
\begin{equation} \label{identity1}
\sd_F(\partial \sigma^n) = \partial \overline{F}*\partial \sigma^{n-\dim(F)}.
\end{equation}  Similarly, if $G \in \Gamma$, then 
\begin{equation} \label{identity2}
\sd_G(\Gamma*\Delta) = \sd_G(\Gamma)*\Delta.
\end{equation}

We prove the claim by backwards induction on $k$.  When $k=d$,  observation \eqref{identity1} establishes the base case of the induction.  For $k<d$,  observation \eqref{identity2} establishes the inductive step.
\end{proof}

\begin{example}   \label{ex:StarToCross}
Consider the boundary of the 3-simplex on vertex set $\{x_0, x_1, x_2, x_3\}$
whose Schlegel diagram is shown in Figure \ref{simplex-fig}, left.  Let
$F_1 = \{x_2,x_3\}$ and $F_2 = \{x_1, x_2, x_3\}$.  The Schlegel diagram of the
subdivision $\sd_{F_1} \circ \sd_{F_2}(\partial \sigma^3)$ is shown in the middle
of Figure \ref{simplex-fig}, and its realization as the Schlegel diagram of the
octahedron is shown on the far right of Figure \ref{simplex-fig}.
\end{example}

\begin{figure}[ht]
\begin{center}
\begin{tikzpicture}
\draw[fill=black] (0,0) circle (.1);
\draw[fill=black] (3,0) circle (.1);
\draw[fill=black] (1.5,2.6) circle (.1);
\draw[fill=black] (1.5,1) circle (.1);
\draw (0,0) -- (3,0) -- (1.5,2.6) -- (1.5,1) -- (0,0);
\draw (0,0) -- (1.5,2.6);
\draw (1.5,1) -- (3,0);
\draw (1.5,.75) node {$x_3$};
\draw (1.5,2.9) node {$x_2$};
\draw (-.25,-.25) node {$x_1$};
\draw (3.25,-.25) node {$x_0$};
\end{tikzpicture}
\qquad
\begin{tikzpicture}
\draw[fill=black] (0,0) circle (.1);
\draw[fill=black] (3,0) circle (.1);
\draw[fill=black] (1.5,2.6) circle (.1);
\draw[fill=black] (1.5,1) circle (.1);
\draw (0,0) -- (3,0) -- (1.5,2.6) -- (1.5,1) -- (0,0);
\draw (0,0) -- (1.5,2.6);
\draw (1.5,1) -- (3,0);
\draw (1.5,.75) node {$x_3$};
\draw (1.5,2.9) node {$x_2$};
\draw (-.25,-.25) node {$x_1$};
\draw (3.25,-.25) node {$x_0$};
\draw (1.1,1.25) -- (0,0);
\draw (1.1,1.25) -- (1.5,2.6);
\draw (1.1,1.25) -- (1.5,1);
\draw (1.1,1.25) -- (1.5,1.8);
\draw (1.5,1.8) -- (3,0);
\draw[fill=black] (1.1,1.25) circle (.1);
\draw (1.1,1) node {\footnotesize{$y_0$}};
\draw[fill=black] (1.5,1.8) circle (.1);
\draw (1.8,1.8) node {\footnotesize{$y_1$}};
\end{tikzpicture}
\qquad
\begin{tikzpicture}
\draw[fill=black] (0,0) circle (.1);
\draw[fill=black] (3,0) circle (.1);
\draw[fill=black] (1.5,2.6) circle (.1);
\draw[fill=black] (1.5,.75) circle (.1);
\draw (0,0) -- (3,0) -- (1.5,.75) -- (0,0);
\draw (1.5,2.6) -- (3,0);
\draw (0,0) -- (1.5,2.6);
\draw (1.5,.75) -- (3,0);
\draw (1.5,.5) node {$y_2$};
\draw (1.5,2.9) node {$x_2$};
\draw (-.25,-.25) node {$x_1$};
\draw (3.25,-.25) node {$x_0$};
\draw (1.1,1.25) -- (0,0);
\draw (1.1,1.25) -- (1.5,2.6);
\draw (1.1,1.25) -- (1.5,.75);
\draw (1.1,1.25) -- (1.9,1.25);
\draw (1.9,1.25) -- (1.5,2.6);
\draw (1.9,1.25) -- (3,0);
\draw (1.9,1.25) -- (1.5,.75);
\draw[fill=black] (1.1,1.25) circle (.1);
\draw (1.1,.9) node {\footnotesize{$y_0$}};
\draw[fill=black] (1.9,1.25) circle (.1);
\draw (1.9,.9) node {\footnotesize{$y_1$}};
\end{tikzpicture}

\end{center}
\caption{Drawing $\C_2$ in $\partial \sigma^3$ by subdividing a flag of faces.}
\label{simplex-fig}
\end{figure}
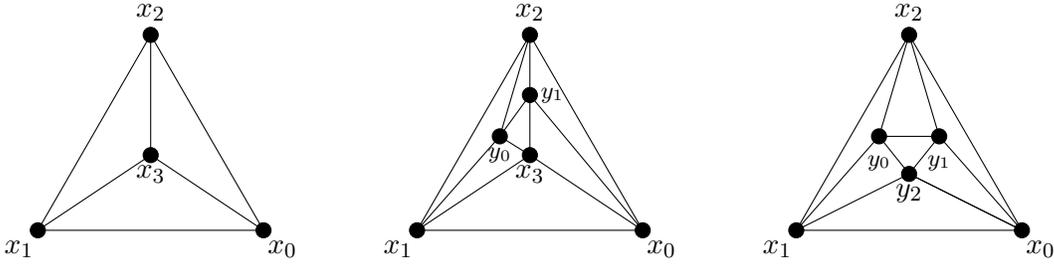

Instead of only working with simplicial complexes, it will be sometimes more
convenient for us to work with cross-polytopal complexes.

\begin{definition} \label{def:cross-poly-compl}
A $(d+1)$-dimensional regular CW-complex $\mathcal{X}$ is a \textit{cross-polytopal complex} if the $d$-dimensional skeleton of $\mathcal{X}$ is a simplicial complex, and the boundary of each $(d+1)$-dimensional cell is isomorphic to $\C_d$. 
\end{definition}

Lemma \ref{lem:StarToCross} allows us to turn any pure balanced simplicial complex into a \emph{cross-polytopal complex}. Indeed, let $\Delta$ be a pure $(d+1)$-dimensional balanced simplicial complex endowed with a coloring $\kappa: V(\Delta)\to\{0,1,\ldots, d+1\}$, and let $\{G_1,\ldots, G_m\}$ be the set of facets of $\Delta$. For each $i$, consider the flag of faces $F^i_1\subset F^i_2\subset \cdots\subset F^i_d$ in $\partial\overline{G_i}$, where $F^i_j$ is the unique face of $G_i$ with $\kappa(F^i_j)=\{d+1-j,\ldots,d+1\}$, and replace $\partial\overline{G_i}$ with the subdivision of $\partial\overline{G_i}$ with respect to ~this flag as explained in Lemma \ref{lem:StarToCross}.

To see that the above construction is consistent, observe that one can alternatively describe it as follows: in the $d$-skeleton of $\Delta$, perform a stellar subdivision at each $d$-face that does not contain color $0$; then perform a stellar subdivision (in the $d$-skeleton) at each $(d-1)$-face that does not contain colors $0$ or  $1$. Continue descending by dimension in this way so that finally we subdivide all edges of $\Delta$ that do not contain colors $0,1,\ldots, d-1$. According to Lemma \ref{lem:StarToCross}, the above procedure replaces the boundary of each simplex $\overline{G_i}$ with the boundary of a cross-polytope. Hence it transforms $\Delta$ into a $(d+1)$-dimensional cross-polytopal complex with each $(d+1)$-simplex $\overline{G_i}$ of $\Delta$ replaced with a $(d+1)$-dimensional cross-polytope $\G_i$.

\begin{definition} \label{diamond-op} 
We refer to the procedure described in the previous paragraph as the \emph{diamond operation} on $\Delta$, and denote the resulting cross-polytopal complex by $\Diamond(\Delta,\kappa)$ or simply by $\Diamond(\Delta)$ if $\kappa$ is understood. Note that each face $F$ of $\Delta$ with $d+1\notin \kappa(F)$ remains a face in $\Diamond(\Delta)$.
\end{definition}

\subsection{Connecting balanced spheres}

We are now ready to provide the proof of Theorem \ref{main-thm} for the case of
spheres.

\begin{theorem} \label{cross-poly->spheres}
Let $\Delta$ be a balanced simplicial complex. If $\Delta$ is a combinatorial
$d$-sphere, then there is a sequence of cross-flips that transforms $\C_d$ into
$\Delta$.
\end{theorem}

\begin{proof}
By Theorem \ref{shellable-ball}, there exists a shellable $(d+1)$-ball whose boundary is $\Delta$. Applying Corollary \ref{balanced->balanced} to this ball produces a balanced $(d+1)$-ball $\Omega$ such that $\partial\Omega=\Delta$ and the balanced coloring of $\Omega$ extends that of $\Delta$. In particular, no vertex of $\Delta$ has color $d+1$.

Since $\Omega$ is a stellar subdivision of a shellable ball, it is shellable by Theorem \ref{subdiv-shellability}. Let $G_0,G_1,\ldots, G_t$ be a shelling order on $\Omega$. Define $\Omega_0:=\overline{G_0}$, and for each $1\leq j \leq t$, 
define
\[
\Omega_j:=\overline{G_0}\cup\cdots\cup\overline{G_j} \quad \text{ and } \quad \Gamma_j := \overline{G_j} \cap \Omega_{j-1}.
\]
By definition of a shelling, $\Gamma_j$ is a pure $d$-dimensional subcomplex of $\partial \overline{G_j}$. Besides, $\Gamma_j \ne \partial \overline{G_j}$, since $\Omega_j$ is a ball and not a sphere.

Next, we apply the diamond operation (Definition \ref{diamond-op}) to $\Omega$. Each of the complexes $\Omega_j$ becomes a cross-polytopal complex $\Diamond(\Omega_j) = \Diamond(\overline{G_0}) \cup \cdots \cup \Diamond(\overline{G_j}).$  Consider the sequence
\[
\C_d \cong \partial(\Diamond(\Omega_0)) \to \partial(\Diamond(\Omega_1)) \to \cdots \to \partial(\Diamond(\Omega_t)) = \Delta,
\]
where $\partial(\Diamond(\Omega_t)) = \partial(\Omega_t) = \Delta$ because no vertex of $\Delta$ has color $d+1$. Let us show that every move in the above sequence is a cross-flip.

Indeed, as $\Omega_j$ is obtained from $\Omega_{j-1}$ by gluing $\overline{G_j}$ along $\Gamma_j \subset \partial \overline{G_j}$, the complex $\Diamond(\Omega_j)$ is obtained from $\Diamond(\Omega_{j-1})$ by gluing $\Diamond(\overline{G_j})$ along $\Diamond(\Gamma_j)$ 
It follows that the transformation $\partial(\Diamond(\Omega_j)) \to \partial(\Diamond(\Omega_{j-1}))$ replaces $\Diamond(\Gamma_j)$ with $\Diamond(\partial \overline{G}_j) \setminus \Diamond(\Gamma_j) \cong \C_d \setminus \Diamond(\Gamma_j)$. Further, $\Diamond(\Gamma_j)$ is shellable because it is obtained from the shellable complex $\Gamma_j$ through stellar subdivision(s); similarly, its complement in $\C_d$ is shellable. Also, stellar subdivisions preserve the property of subcomplexes being induced; therefore $\Diamond(\Gamma_j)$ is an induced subcomplex of $\partial \Diamond(\Omega_{j-1})$ since $\Gamma_j$ was an induced subcomplex of $\partial \Omega_{j-1}$. Thus all conditions in the definition of a cross-flip are fulfilled, and we have a sequence of cross-flips taking $\C_d$ to $\Delta$.
\end{proof}

\begin{corollary} \label{all-spheres}
If $\Delta$ and $\Gamma$ are balanced combinatorial $d$-spheres, then there is a
sequence of cross-flips that transforms $\Delta$ into $\Gamma$. Moreover, there 
is a sequence of cross-flips that transforms a given coloring of $\Delta$ into a 
given coloring of $\Gamma$.
\end{corollary}

The second statement of the corollary is a non-trivial strengthening of the 
first one. To see this, consider different colorings of the bipyramid over a 
$(2n)$-gon. To transform one coloring into another, ``shrink'' the equator of a 
bipyramid by repeatedly applying the inverse of the second move in Figure 
\ref{fig:CrossFlips2} until we obtain an octahedron; then repeatedly ``expand'' the other equator of the octahedron to 
obtain a different coloring of the same bipyramid. Without specifying colorings, it is not necessary to apply any cross-flips.

\begin{remark}
\label{rem:3Suffice}
It follows from the proof of Theorem \ref{cross-poly->spheres} that it suffices to consider cross-flips of the form $\Diamond(\Gamma) \rightarrow \C_d\setminus \Diamond(\Gamma)$ where $\Gamma \subseteq \partial\sigma^{d+1}$ is a $d$-ball.  For example, in dimension two, these are exactly the cross-flips that appear in the first two rows of Figure \ref{fig:CrossFlips2} (see also, Figure \ref{simplex-fig}). Since the first move in the second row is trivial, the other three moves suffice to connect any two balanced triangulations of a surface.

In higher dimensions, we do not have an explicit description of a minimal set of cross-flips that suffice to connect any pair of balanced, PL homeomorphic combinatorial manifolds.
\end{remark}

\section{Background, definitions, and elementary properties -- part II}
In this section we discuss several additional definitions and results needed to prove Theorems \ref{color-thm} and \ref{main-thm} in their full generality.  We start by slightly generalizing the notion of elementary shellings from \S\ref{subsect:shellability1}.

\subsection{Bistellar flips and elementary pseudo-cobordisms}
\label{sec:ElemPseudo}
Recall that a bistellar flip from a simplicial complex $\Delta$ to a  simplicial complex $\Delta'$ replaces $\overline{A} *  \partial\overline{B}$ with $\partial\overline{A} * \overline{B}$ for simplices $A \in \Delta$ and $B \notin \Delta$ such that $\lk_{\Delta}(A) = \partial \overline{B}$ (see \S 2.1).   We call the simplicial complex $\Delta \cup (\overline{A}*\overline{B})$ an \emph{elementary pseudo-cobordism} between $\Delta$ and $\Delta'$. Since the inverse of a bistellar flip is itself a bistellar flip and $\Delta' \cup (\overline{A} * \overline{B}) = \Delta \cup (\overline{A} * \overline{B})$, this definition is symmetric with respect to $\Delta$ and $\Delta'$.

Geometrically (and somewhat informally) we can say that we are gluing the $(d+1)$-dimensional simplex on $A \cup B$ ``on top'' of the complex $\Delta$ along the subcomplex $\overline{A}*\partial\overline{B}$. The  complex $\Delta'$ is what we see when we look at $\Delta \cup (\overline{A} * \overline{B})$  ``from above.'' Figure \ref{fig:Cobord2} illustrates an example for $d=2$ (the view from ``above'' is shown on the right).

\begin{figure}[htb]
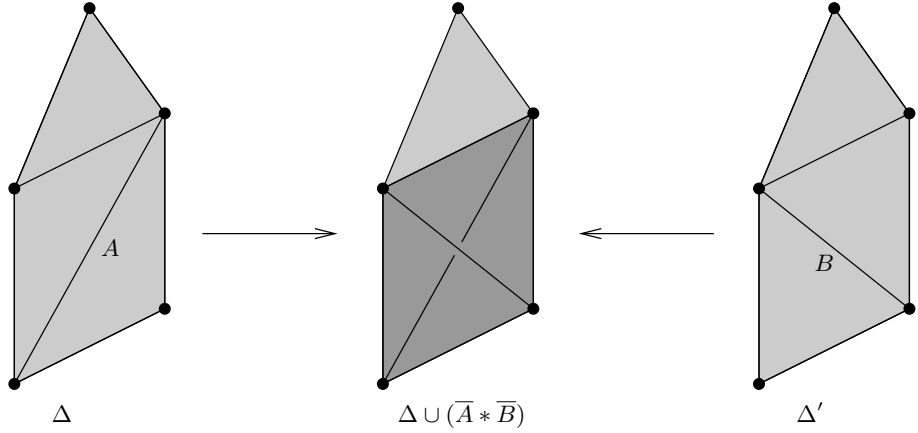
\caption{An elementary pseudo-cobordism.}
\label{fig:Cobord2}
\end{figure}

Performing more bistellar flips will result in gluing more $(d+1)$-dimensional  simplices on top of~$\Delta$. This can lead to a cell complex that is no  longer a simplicial complex. For example, this is the case when we undo an edge flip in dimension $2$. Therefore we need to enlarge our scope to study the more general class of simplicial posets.

\subsection{Simplicial posets}
\label{sec:SimPos}
A \textit{simplicial poset} is a finite poset $P$ with a unique minimal element, $\emptyset$, such that the interval $\left[\emptyset,\tau\right]$ is isomorphic to a Boolean lattice for each $\tau \in P$.  A simplicial poset is naturally graded by declaring that $\rk(\tau) = k$ if $\left[\emptyset,\tau\right]$ is a Boolean lattice of rank $k$.  Geometrically, a simplicial poset can be realized as a regular CW-complex, denoted $|P|$, by inserting a $(\rk(\tau)-1)$-dimensional simplex for each face $\tau \in P$; the faces of this simplex correspond to the elements of $\left[\emptyset,\tau\right]$. For this reason, simplicial posets are also referred to  as \textit{simplicial cell complexes}, \textit{Boolean cell complexes}, or \textit{pseudo-simplicial complexes} in the literature.  The maximal elements of $P$ are called \textit{facets}.

The key difference between a simplicial complex and a simplicial poset is that two faces in a simplicial poset may intersect along an arbitrary subcomplex of their boundaries, whereas two faces in a simplicial complex intersect along a (possibly empty) face. On the other hand, the assumption that each interval $\left[\emptyset, \tau\right]$ is a Boolean lattice means that all attaching maps are injective. In particular, the $1$-skeleton of a simplicial cell complex contains no loops, but may contain multiple edges. Thus it is still natural to study vertex colorings on simplicial posets, and we say that a $d$-dimensional simplicial poset is \textit{balanced} if its $1$-skeleton admits a proper vertex $(d+1)$-coloring.

Since each cell of a simplicial poset is a simplex, many definitions pertaining to simplicial complexes have natural extensions to simplicial posets. For instance, if $(P, \leq)$ is a simplicial poset and $\tau$ is a face of $P$, then the \textit{deletion} of $\tau$ from $P$ as well as the \textit{star}  and the \textit{link} of $\tau$ in $P$ are defined as follows:
\[
P\setminus \tau := \{\sigma\in P \, : \, \tau\not\leq \sigma\}, \quad \lk_P(\tau):=\{\sigma\in P \, : \, \tau\leq \sigma\}, \quad \text{and} \quad \st_P(\tau):=\bigcup_{\sigma \in \lk_P(\tau)} [\emptyset, \sigma]. 
\]
The deletions and stars of simplicial posets are themselves simplicial posets, and so are the links. However, it is worth noting that $\lk_{P}(\tau)$ is a simplicial poset with minimal element $\tau$, and that, in contrast to the setting of simplicial complexes, it is not naturally a subcomplex of $P$ (e.g., consider the link of a vertex in a graph made of two vertices and two edges joining them).

Similarly, if $(P,\leq)$ and $(P',\leq')$ are two simplicial posets, then their \text{join}, $(P\ast P', \preceq)$, is the simplicial poset on the set $P\times P'$ with the order relation $(\tau,\tau')\preceq (\sigma,\sigma')$ if $\tau\leq \sigma$ and $\tau'\leq'\sigma'$.  The \textit{stellar subdivision} of a simplicial poset $P$ at a face $\tau$ is defined as
\[
\sd_\tau(P):= \left(P\setminus \tau\right) \cup \left([\emptyset,a] \ast [\emptyset, \tau) \ast \lk_P(\tau)\right), 
\]
where $a$ is a new element of rank $1$ (and the elements of $P\setminus \tau$ and $\{a\} \ast [\emptyset, \tau) \ast \lk_P(\tau)$ are incomparable).
Somewhat informally, to construct $\sd_\tau(P)$, we perform a stellar subdivision at $\tau$ for each simplex of $P$ that contains $\tau$. With these definitions in hand, the statement and the proof of Theorem \ref{ext-coloring}, and hence also of Corollary \ref{balanced->balanced}, mechanically extend to the case of simplicial posets.

In analogy with relative simplicial complexes, a \textit{relative simplicial poset} is a pair of simplicial posets $(P,Q)$ such that $Q \subseteq P$ is a lower order ideal (i.e., if $\tau \in Q$ and $\sigma < \tau$, then $\sigma \in Q$).  The \textit{dimension} of a relative simplicial poset is $\dim(P,Q) = \max\{ \dim(\tau) \mid \tau \in P \setminus Q\}$.  A \textit{facet} in a relative simplicial poset is a face $\sigma \in P \setminus Q$ that is not properly contained in any other faces of $P \setminus Q$.  A relative simplicial poset is \textit{pure} if all of its facets have the same dimension.

As in the case of simplicial complexes, a pure simplicial poset is \textit{shellable} if its facets can be ordered $F_1, F_2, \ldots, F_t$ such that for all $j$ there exists a face $r(F_j) \leq F_j$ such that $$[\emptyset, F_j] \setminus \bigcup_{i<j} [\emptyset, F_i] = [r(F_j),F_j].$$ 
More generally, a relative simplicial poset $(P,Q)$ is \textit{shellable} if it is pure and its facets can be ordered $F_1, F_2, \ldots, F_t$ such that for all $1 \leq j \leq t$, there exists a face $r(F_j) \leq F_j$ such that
\begin{equation}
\label{eqn:RelShellPoset}
[\emptyset, F_j] \setminus \Bigg( \bigcup_{i < j} [\emptyset, F_i] \cup Q \Bigg) = [r(F_j), F_j].
\end{equation}

There are many equivalent definitions of a shellable simplicial complex.  The following lemma is a relative analogue of one of these equivalent definitions.

\begin{lemma} \label{alt-def-rel-shelling}
Let $(P,Q)$ be a shellable relative simplicial poset with shelling order $F_1, F_2, \ldots, F_t$.  For any $j$ and any vertex $v \leq F_j$, 
$$
v \le r(F_j) \qquad \text{if and only if } \qquad F_j \setminus v \in \bigcup_{i < j} [\emptyset, F_i] \cup Q,
$$
with the convention that $F_j \setminus v$ is the unique face covered by $F_j$ that does not contain $v$.
\end{lemma}

\begin{proof}
Note that 
\begin{eqnarray*}
v \leq r(F_j) & \Leftrightarrow& r(F_j) \not\leq F_j \setminus v \\
&\Leftrightarrow& F_j \setminus v \notin [r(F_j), F_j] \\
&\Leftrightarrow& F_j \setminus v \in \bigcup_{i < j} [\emptyset, F_i] \cup Q.
\end{eqnarray*}
The equivalence in the last line comes from the definition of a shelling since $F_j \setminus v \in [\emptyset, F_j]$.
\end{proof}

We also extend the notion of a \textit{cross-polytopal complex} (see Definition \ref{def:cross-poly-compl}) to mean a $(d+1)$-dimensional regular CW complex $\mathcal{X}$ such that (i) the $d$-dimensional skeleton of $\mathcal{X}$ is a \textit{simplicial poset}, and (ii) the boundary of each $(d+1)$-dimensional cell is isomorphic to $\C_d$. As a result, the diamond operation of Definition \ref{diamond-op} can now be applied to any balanced simplicial poset.

\subsection{Pseudo-cobordisms, relative shellings, and bistellar flips}

Now we are prepared to make the key idea presented at the beginning of this section -- that a sequence of bistellar flips on a $d$-dimensional simplicial complex can be encoded by successively attaching $(d+1)$-simplices -- more precise.  We begin with the following definitions.

\begin{definition}
Let $\Omega$ be a $(d+1)$-dimensional simplicial poset.  We say that $\Omega$ is a \textit{nonpure pseudomanifold} if every $d$-face in $\Omega$ is contained in at most two (and possibly zero) $(d+1)$-faces of $\Omega$.  The \textit{pseudoboundary} of a nonpure pseudomanifold $\Omega$, denoted $\widetilde{\partial} \Omega$, is the subcomplex of $\Omega$ induced by all $d$-faces that are contained in zero or one $(d+1)$-faces of $\Omega$.
\end{definition}

\begin{definition}
\label{dfn:PseudoCobord}
Let $\Delta$ and $\Delta'$ be closed combinatorial $d$-manifolds (in particular, $\Delta$ and $\Delta'$ are simplicial complexes).    A \textit{pseudo-cobordism} between $\Delta$ and $\Delta'$ is a $(d+1)$-dimensional nonpure pseudomanifold $\Omega$, together with a pair of simplicial embeddings $\varphi: \Delta \hookrightarrow \Omega$ and $\varphi': \Delta' \hookrightarrow \Omega$ such that
\begin{enumerate}
\item $\varphi(\Delta) \cup \varphi'(\Delta') = \widetilde{\partial}\Omega$ and
\item a $d$-face $F \in \Omega$ belongs to $\varphi(\Delta) \cap \varphi'(\Delta')$ if and only if $F$ is not contained in any $(d+1)$-faces of $\Omega$.
\end{enumerate}
\end{definition}

An example of a pseudo-cobordism is the elementary pseudo-cobordism described in Section~\ref{sec:ElemPseudo}. Note that the second condition in Definition \ref{dfn:PseudoCobord} implies that each of the pairs $(\Omega, \phi(\Delta))$ and $(\Omega, \phi'(\Delta'))$ is a pure relative simplicial poset, that is, $\Omega$ can be obtained from $\Delta$ (or from $\Delta'$) by attaching a number of $(d+1)$-dimensional simplices. As always, we define \emph{shellable} pseudo-cobordisms as those where the attaching  of $(d+1)$-simplices can be done in an especially nice way:

\begin{definition} \label{def:shellablePseudoCobord}
Let $\Delta$ and $\Delta'$ be closed combinatorial $d$-manifolds. A \textit{pseudo-cobordism} $\Omega$ between $\Delta$ and $\Delta'$ is \textit{shellable} if its $(d+1)$-dimensional faces can be ordered $F_1, F_2, \ldots, F_t$ such that
\begin{enumerate}
\item $F_1, F_2, \ldots, F_t$ is a shelling order on $(\Omega, \varphi(\Delta))$ and 
\item $F_t, F_{t-1}, \ldots, F_1$ is a shelling order on $(\Omega, \varphi'(\Delta'))$.
\end{enumerate}
\end{definition}

Shellable pseudo-cobordisms with a unique $(d+1)$-face are precisely the elementary pseudo-cobordisms from Section \ref{sec:ElemPseudo}.

\begin{definition}
Let $(\Omega, \phi, \phi')$ be a pseudo-cobordism between $\Delta$ and 
$\Delta'$, and $(\Omega', \psi', \psi'')$ be a pseudo-cobordism between 
$\Delta'$ and $\Delta''$. Then the triple,
\begin{equation}
\label{eqn:GluingCobord}
(\Omega \cup_{(\phi',\psi')} \Omega', \phi, \psi'')
\end{equation}
where $\Omega \cup_{(\phi',\psi')} \Omega'$ is the result of gluing $\Omega$ and $\Omega'$ along the isomorphic subcomplexes $\phi'(\Delta')$ and $\psi'(\Delta')$, is called the \emph{composition} of pseudo-cobordisms $(\Omega, \phi, \phi')$ and $(\Omega', \psi', \psi'')$.  We denote by $\phi'(\Delta')\sim \psi'(\Delta')$ the image of $\phi'(\Delta')$ (and also of $\psi'(\Delta')$) in $\Omega \cup_{(\phi',\psi')} \Omega'$.
\end{definition}

\begin{lemma}
\label{lem:GluingCobord}
The composition of pseudo-cobordisms is again a pseudo-cobordism. Further, the composition of shellable pseudo-cobordisms is a shellable pseudo-cobordism.
\end{lemma}
\begin{proof}
We start by showing that \eqref{eqn:GluingCobord} is a pseudo-cobordism between $\Delta$ and $\Delta''$. To do so, for a $d$-face $\sigma \in \Omega$, we let $\deg_\Omega \sigma$ denote the number of $(d+1)$-faces of $\Omega$ containing $\sigma$. The definition of a pseudo-cobordism between $\Delta$ and $\Delta'$ is then equivalent to the following:
\begin{itemize}
\item each $d$-face has degree at most $2$;
\item the $d$-faces of degree $1$ are the $d$-faces of the symmetric difference of $\phi(\Delta)$ and $\phi'(\Delta')$;
\item the $d$-faces of degree $0$ are the $d$-faces of the intersection of $\phi(\Delta)$ and $\phi'(\Delta')$.
\end{itemize}

The degree of a $d$-face $\sigma \in \Omega'' := \Omega \cup_{(\phi',\psi')} \Omega'$ depends on where $\sigma$ lies with respect to the images of $\Delta$, $\Delta'$, and $\Delta''$ in $\Omega''$. Namely,
\[
\deg_{\Omega''} \sigma =
\begin{cases}
\deg_\Omega \sigma, &\text{if } \sigma \in \Omega \setminus \phi'(\Delta'),\\
\deg_{\Omega'} \sigma, &\text{if } \sigma \in \Omega' \setminus \psi'(\Delta'),\\
\deg_\Omega \sigma + \deg_{\Omega'} \sigma, &\text{if } \sigma \in \phi'(\Delta') \sim \psi'(\Delta').
\end{cases}
\]
It follows that $\deg_{\Omega''} \sigma \le 2$.

Further, we have $\deg_{\Omega''} \sigma = \deg_{\Omega} \sigma = 1$ if and only if $\sigma$ belongs to one of the sets
\[
\phi(\Delta) \setminus \phi'(\Delta'), \quad \phi(\Delta) \cup \phi'(\Delta') \setminus \psi''(\Delta'').
\]
It is easy to see that the union of these sets is $\phi(\Delta) \setminus \psi''(\Delta'')$. Similarly, we have $\deg_{\Omega''} \sigma = \deg_{\Omega'} \sigma = 1$ if and only if $\sigma \in \psi''(\Delta'') \setminus \phi(\Delta)$. Hence $\deg_{\Omega''} \sigma = 1$ if and only if $\sigma$ belongs to the symmetric difference of $\phi(\Delta)$ and $\psi''(\Delta'')$.

Finally, if $\sigma \in \Omega$ and $\deg_{\Omega''} \sigma = 0$, then $\sigma \in \phi'(\Delta') \sim \psi'(\Delta')$. Hence $\sigma \in \Omega \cap \Omega'$ and $\deg_{\Omega} \sigma = \deg_{\Omega'} \sigma = 0$. This occurs if and only if
\[
\sigma \in \left( \phi(\Delta) \cap \phi'(\Delta') \right) \cap \left( \psi'(\Delta') \cap \psi''(\Delta'') \right) = \phi(\Delta) \cap \psi''(\Delta'').
\]
For the last equality, note that if $\sigma$ belongs to $\phi(\Delta) \cap \psi''(\Delta'')$, then $\Omega$ and $\Omega'$ were glued along $\sigma$, which implies that $\sigma \in \phi'(\Delta') \sim \psi'(\Delta')$.

It remains to show that shellability of pseudo-cobordisms $\Omega$ and $\Omega'$ implies that of $\Omega''$. Let $F_1, \ldots, F_t$ and $G_1, \ldots, G_s$ be shelling orders of $(\Omega, \varphi(\Delta))$ and $(\Omega', \psi'(\Delta'))$ respectively,  such that their reversals are shelling orders of $(\Omega, \varphi'(\Delta'))$ and $(\Omega', \psi''(\Delta''))$, respectively.
We claim that $F_1, \ldots, F_t, G_1, \ldots, G_s$ is a shelling order of $(\Omega'', \varphi(\Delta))$.

Indeed, for $F_1, \ldots, F_t$ there is nothing to check since it is already a shelling order on $(\Omega, \varphi(\Delta)$ that is not affected by the faces of $\Omega' \setminus \Omega$.
As for $G_1, \ldots, G_s$, we have
\[
[\emptyset, G_j] \setminus \Bigg( \bigcup_{k=1}^t [\emptyset, F_k] \cup \bigcup_{i<j} [\emptyset, G_i] \cup \varphi(\Delta) \Bigg) 
= [\emptyset, G_j] \setminus \Bigg( \bigcup_{i<j} [\emptyset, G_i] \cup \psi'(\Delta') \Bigg)
= [r(G_j),G_j],
\]
since the only faces of $\Omega$ that belong to $[\emptyset, G_j]$ are those in $\varphi'(\Delta') \sim \psi'(\Delta')$.  The proof that $G_s, \ldots, G_1, F_t, \ldots, F_1$ is a shelling order of $(\Omega'', \psi''(\Delta''))$ is identical.
\end{proof}

\begin{proposition}
\label{prp:ShellCompos}
A pseudo-cobordism is shellable if and only if it can be represented as a composition of elementary pseudo-cobordisms.
\end{proposition}
The proof of this proposition is based on a sequence of lemmas and is therefore postponed to the end of the section.  The main result of this section is the following direct corollary of Proposition~\ref{prp:ShellCompos}.

\begin{theorem}
\label{thm:BistCobord}
Let $\Delta$ and $\Delta'$ be closed combinatorial $d$-manifolds.  There exists a shellable pseudo-cobordism between $\Delta$ and $\Delta'$ if and only if $\Delta$ and $\Delta'$ are bistellar equivalent.
\end{theorem}
\begin{proof}
Every shellable pseudo-cobordism is a composition of elementary pseudo-cobordisms. At the same time, two complexes related by an elementary pseudo-cobordism differ by a bistellar flip.  Hence the ends of a shellable pseudo-cobordism are bistellarly equivalent. In the opposite direction, a sequence of bistellar flips produces a sequence of elementary pseudo-cobordisms, which can be composed to produce a shellable pseudo-cobordism.
\end{proof}

Let $\Omega$ be a shellable pseudo-cobordism between closed combinatorial $d$-manifolds $\Delta$ and $\Delta'$, with shelling order $F_1, \ldots, F_t$. 
For $0 \leq j \leq t$, define
\[
\Omega_j:= \bigcup_{i \leq j} [\emptyset, F_i] \cup \varphi(\Delta),
 \qquad
\Omega'_j:= \bigcup_{k > j} [\emptyset, F_k] \cup \varphi'(\Delta'),
\qquad \mbox{and} \qquad
\Delta_j := \Omega_j \cap \Omega'_j.
\]
In particular, $\Omega_0 = \varphi(\Delta)$, $\Omega'_t = \varphi'(\Delta')$, and $\Omega_t = \Omega'_0 = \Omega$, which implies $\Delta_0 = \phi(\Delta)$ and $\Delta_t = \phi'(\Delta')$. Our goal will be to show that $\Omega$ is a composition of elementary pseudo-cobordisms between $\Delta_{j-1}$ and $\Delta_j$ for $j = 1, \ldots, t$.

Let $B_j$ be the restriction face of facet $F_j$ in the shelling order on $(\Omega, \varphi(\Delta))$ and let $A_j$ be the restriction face of facet $F_j$ in the shelling order on $(\Omega, \varphi'(\Delta'))$.  More concretely, for all $1 \leq j \leq t$, $A_j$ and $B_j$ are the unique faces such that
\[
[\emptyset, F_j] \setminus \Omega_{j-1} = [B_j, F_j] \qquad
\mbox{and} \qquad
[\emptyset, F_j] \setminus \Omega'_j = [A_j, F_j].
\]

\begin{lemma} \label{Aj-Bj-meet-join}
For all $1 \le j \le t$, $\emptyset$ is the greatest lower bound of $A_j$ and $B_j$, and $F_j$ is the least upper bound of $A_j$ and $B_j$ in $[\emptyset, F_j]$.  
\end{lemma}

\begin{proof}
We will write $\text{glb}(A_j,B_j)$ and $\text{lub}(A_j,B_j)$ respectively to denote the greatest lower bound and least upper bound of $A_j$ and $B_j$ in $[\emptyset,F_j]$. Note that $\text{glb}(A_j,B_j)$ and $\text{lub}(A_j,B_j)$ in $[\emptyset, F_j]$ exist because $[\emptyset, F_j]$ is isomorphic to a Boolean lattice. Furthermore, by Lemma \ref{alt-def-rel-shelling},
\[
B_j = \{v \le F_j \mid F_j \setminus v_j \in \Omega_{j-1}\} \qquad \text{and} \qquad A_j = \{v \le F_j \mid F_j \setminus v_j \in \Omega'_j\}.
\]

First we show that $\text{glb}(A_j,B_j) = \emptyset$.  Suppose instead that there exists a vertex $v \leq F_j$ such that $v \leq A_j$ and $v \leq B_j$, that is
\[
F_j \setminus v \in \Omega_{j-1} \qquad
\text{and} \qquad
 F_j \setminus v \in \Omega'_j.
\]
Since $\Omega_{j-1}$ and $\Omega'_j$ have no common $(d+1)$-faces, and since $F_j \setminus v$ is covered by the $(d+1)$-face $F_j \notin \Omega_{j-1} \cup \Omega'_j$, we infer that
\[
\deg_{\Omega} (F_j \setminus v) = \deg_{\Omega_{j-1}} (F_j \setminus v) + \deg_{\Omega'_j} (F_j \setminus v) + 1.
\]
Also, since $\Omega$ is a nonpure pseudomanifold,  $\deg_{\Omega}(F_j \setminus v) \le 2$. Thus, it follows that the degree of $F_j \setminus v$ in at least one of the posets $\Omega_{j-1}$, $\Omega'_j$ must be equal to zero. Without loss of generality, assume $\deg_{\Omega_{j-1}}(F_j \setminus v) = 0$. Since $F_j \setminus v \in \Omega_{j-1}$, we conclude that $F_j \setminus v$ belongs to $\phi(\Delta)$. The fact that $\phi(\Delta)$ is a part of the pseudoboundary of $\Omega$, then yields that $\deg_\Omega(F_j \setminus v) \le 1$, which, in turn, implies that the degree of $F_j \setminus v$ in $\Omega'_j$ also vanishes. By the second condition in Definition \ref{dfn:PseudoCobord} we then obtain that $\deg_{\Omega}(F_j \setminus v) = 0$. This however contradicts the assumption that $F_j \setminus v \le F_j$.

Next, we show that $F_j$ is the least upper bound of $A_j$ and $B_j$ in $[\emptyset, F_j]$.  Suppose instead that there exists $v \leq F_j$ such that $v \not\leq A_j$ and $v \not\leq B_j$.
Then
\[
F_j \setminus v \notin \Omega_{j-1} \qquad
\text{and} \qquad
 F_j \setminus v \notin \Omega'_j,
\]
which implies that $\deg_\Omega(F_j \setminus v) = 1$. Therefore, either $F_j \setminus v$ belongs to $\phi(\Delta)$ or it belongs to $\phi'(\Delta')$; consequently $F_j \setminus v$ is an element of $\Omega_{j-1}$ or $\Omega'_j$. This contradicts our assumption.\end{proof}

As a consequence of Lemma \ref{Aj-Bj-meet-join}, note that the faces $A_j$ and $B_j$ are proper and nonempty faces of $F_j$.  Indeed, $A_j = \emptyset$ if and only if $B_j = F_j$, so it suffices to show that $A_j \neq \emptyset$.  This follows from the definition of $A_j$ since $\emptyset \in \varphi'(\Delta')$.  

\begin{lemma} \label{Aj-Bj-in-Delta}
For all $1 \leq j \leq t$, $A_j \in \Delta_{j-1} \setminus \Delta_j$.
\end{lemma}

\begin{proof}
Observe that by definition of $A_j$, $A_j \notin \Omega'_j$, and hence $A_j \notin \Delta_j$. On the other hand, $A_j \in \Omega'_{j-1}$ because $A_j \in [0, F_j] \subset \Omega'_{j-1}$, and $A_j \in \Omega_{j-1}$ because
\[
[B_j, F_j] = [0, F_j] \setminus \Omega_{j-1}, \qquad \text{while} \quad A_j \in [0,F_j] \quad \text{and} \quad A_j \notin [B_j,F_j].
\]
(The latter holds since $B_j \ne \emptyset$ and $A_j \wedge B_j = \emptyset$.) Hence $A_j \in \Omega_{j-1} \cap \Omega'_{j-1} = \Delta_{j-1}$. \end{proof}

\begin{lemma} \label{Aj-Bj-link}
For all $1 \leq j \leq t$, $\lk_{\Delta_{j-1}}(A_j) = [A_j, F_j)$.
\end{lemma}

\begin{remark}
Recall that the link of a face in a simplicial poset is the upper order ideal generated by that face.  Thus $\lk_{\Delta_{j-1}}(A_j)$ is the set of all faces in $\Delta_{j-1}$ that contain $A_j$.
\end{remark}

\begin{proof}
First let $\tau \in [A_j,F_j)$.
Then  $\tau \in \Omega'_{j-1}$, since $[\emptyset, F_j] \subset \Omega'_{j-1}$. Also $\tau \in \Omega_{j-1}$ since
\[
[B_j,F_j] = [\emptyset, F_j] \setminus \Omega_{j-1} \quad \text{and} \quad B_j \not< \tau.
\]
(The latter holds because, by Lemma \ref{Aj-Bj-meet-join}, $\tau$ cannot be an upper bound for $A_j$ and $B_j$.)
Therefore, $\tau\in\Delta_{j-1}$. As $A_j\leq\tau$, it follows that $\tau\in\lk_{\Delta_{j-1}}(A_j)$.

Conversely, suppose $\sigma \in \lk_{\Delta_{j-1}}(A_j)$.
This means that $\sigma \in \Delta_{j-1}$ and $\sigma \ge A_j$. Our goal is to show that $\sigma < F_j$. First, $\sigma \in \Omega'_{j-1}$, because $\Delta_{j-1} \subset \Omega'_{j-1}$. On the other hand, $\sigma \notin \Omega'_j$ because $\sigma \ge A_j \notin \Omega'_j$. Since $\Omega'_{j-1} = [\emptyset, F_j] \cup \Omega'_j$, it follows that $\sigma \in [\emptyset, F_j]$. Also $\sigma \ne F_j$ because $F_j \notin \Delta_{j-1}$. Thus $\sigma < F_j$, and we are done. \end{proof}

\begin{lemma} \label{Aj-Bj-flip}
For all $1 \leq j \leq t$, $$\Delta_j = (\Delta_{j-1} \setminus A_j) \cup [B_j,F_j).$$
\end{lemma}

\begin{proof}
By the definition of $\Omega$ being shellable, 
$$\Omega_j = \Omega_{j-1} \cup [\emptyset, F_j] = \Omega_{j-1} \sqcup [B_j,F_j] \quad
\mbox{and} \quad
\Omega'_{j-1} = \Omega'_{j} \cup [\emptyset, F_j] = \Omega'_{j} \sqcup [A_j,F_j].$$  (Here, $\sqcup$ denotes disjoint union.)
Thus
\begin{multline*}
\Delta_j = \Omega_j \cap \Omega'_j = (\Omega_{j-1} \sqcup [B_j,F_j]) \cap (\Omega'_{j-1} \setminus [A_j,F_j]) \\
= \big((\Omega_{j-1} \cap \Omega'_{j-1}) \cup [B_j,F_j]\big) \setminus [A_j, F_j] = (\Delta_{j-1} \setminus A_j) \cup [B_j,F_j),
\end{multline*}
where we have used Lemma \ref{Aj-Bj-meet-join} and the fact that $F_j \in \Omega'_{j-1}$.
\end{proof}

\begin{proof}[Proof of Proposition \ref{prp:ShellCompos}]
An elementary pseudo-cobordism is shellable by definition. Hence a composition of elementary pseudo-cobordisms is shellable by Lemma \ref{lem:GluingCobord}.

Conversely, assume that $\Omega$ is a shellable pseudo-cobordism.
Earlier in this section we defined subposets $\Delta_j \subset \Omega_j \subset \Omega$.
We now use induction on $j$ to show that $\Delta_j$ is a simplicial complex, that
$\Omega_j$ is a pseudo-cobordism between $\Delta$ and $\Delta_j$, and that it can be represented as a composition of elementary ones. Assume this is true for $j-1$. By  Lemmas \ref{Aj-Bj-link} and \ref{Aj-Bj-meet-join}, $\lk_{\Delta_{j-1}}(A_j) = \partial\overline{B_j}$,  where $\Delta_{j-1}$ is viewed as a simplicial complex. 
Together with Lemma \ref{Aj-Bj-flip} this implies that $\Delta_j$ is obtained from $\Delta_{j-1}$ by  a bistellar flip. In particular, $\Delta_j$ is a simplicial complex. The elementary pseudo-cobordism corresponding to this flip is $\Delta_{j-1} \cup [0,F_j]$, and its composition with the pseudo-cobordism $\Omega_{j-1}$ is $\Omega_j = \Omega_{j-1} \cup [0,F_j]$. Since, by the inductive assumption, $\Omega_{j-1}$ is a composition of elementary pseudo-cobordisms, so is $\Omega_j$. For $j=t$ we obtain a decomposition of $\Omega$ into elementary pseudo-cobordisms. \end{proof}

\section{Balanced PL manifolds}

In this section, we prove our main result, Theorem \ref{main-thm}, asserting that two balanced closed combinatorial $d$-manifolds $\Delta$ and $\Gamma$ are PL homeomorphic if and only if there is a sequence of cross-flips that transforms $\Delta$ into $\Gamma$.  We also prove Theorem \ref{color-thm} asserting that when $m \geq d+2$, two closed combinatorial $d$-manifolds that are properly $m$-colored can be connected through a sequence of bistellar flips so that the colors of the vertices are preserved under every flip.  The proofs of these results require a number of intermediate results, and so we begin this section with an outline of the proofs, which we hope will help to motivate the results in the remainder of this section.

\subsection{Proof of Theorems \ref{color-thm} and \ref{main-thm}} \label{Section:proof-of-main-theorem}

Let $m \geq d+2$, and suppose $\Delta$ and $\Gamma$ are properly $m$-colored closed combinatorial $d$-manifolds that are PL homeomorphic.  Constructing a sequence of bistellar flips that connects $\Delta$ to $\Gamma$ and preserves their proper $m$-colorings requires two steps.
\begin{enumerate}
\item First, we show that it is possible to construct a shellable pseudo-cobordism $(\Omega, \varphi, \psi)$ between $\Delta$ and $\Gamma$ such that $\varphi(\Delta) \cap \psi(\Gamma) = \{\emptyset\}$. This is Corollary \ref{step2} below.  
\item \label{Step2} Next, since $\varphi(\Delta) \cap \psi(\Gamma) = \{\emptyset\}$, the subcomplex $\varphi(\Delta) \cup \psi(\Gamma) \subseteq \Omega$ is properly $m$-colorable.  By Corollary \ref{balanced->balanced}, it is possible to subdivide $\Omega$ away from the subcomplex $\varphi(\Delta) \cup \psi(\Gamma)$ to obtain a new pseudo-cobordism $(\Omega', \varphi', \psi')$ between $\Delta$ and $\Gamma$ such that $\Omega'$ is properly $m$-colorable.  We show in Proposition \ref{subdivision-preserves-shellability} below that $\Omega'$ is also a shellable pseudo-cobordism between $\Delta$ and $\Gamma$.
\item Therefore, the shelling on $(\Omega',\Delta)$ encodes a sequence of bistellar flips from $\Delta$ to $\Gamma$ such that each intermediate complex is properly $m$-colored.
\end{enumerate}

When $\Delta$ and $\Gamma$ are balanced, the proof of Theorem \ref{main-thm} follows along similar lines.
\begin{enumerate}
\item  As above, we begin by constructing a shellable pseudo-cobordism $(\Omega, \varphi, \psi)$ between $\Delta$ and $\Gamma$ such that $\varphi(\Delta) \cap \psi(\Gamma) = \{\emptyset\}$.

\item Next, since $\varphi(\Delta) \cap \psi(\Gamma) = \{\emptyset\}$, the subcomplex $\varphi(\Delta) \cup \psi(\Gamma)$ is properly $(d+1)$-colorable.  Again, by Corollary \ref{balanced->balanced}, we may subdivide $\Omega$ away from $\varphi(\Delta) \cup \psi(\Gamma)$ to obtain a new pseudo-cobordism $(\Omega', \varphi', \psi')$ between $\Delta$ and $\Gamma$ such that $\Omega'$ is also balanced.  We show in Proposition \ref{subdivision-preserves-shellability} below that $\Omega'$ is also a shellable pseudo-cobordism between $\Delta$ and $\Gamma$.

\item To complete the proof, we apply the diamond operation to $\Omega'$ and proceed as in the proof of Theorem \ref{cross-poly->spheres}.  The complex $\Diamond(\Omega')$ is a cross-polytopal complex; the shelling order on the facets of $\Omega'$ induces an order on the cross-polytopal cells of $\Diamond(\Omega')$; and this order encodes a sequence of cross-flips between $\Delta$ and $\Gamma$.
\end{enumerate}

\subsection{Constructing a pseudo-cobordism with disjoint ends}

Our first goal is to describe an algorithm that takes as its input a pair of closed combinatorial $d$-manifolds, $\Delta$ and $\Gamma$, that are PL homeomorphic and outputs a shellable pseudo-cobordism, $(\Omega, \varphi, \varphi')$, from $\Delta$ to $\Gamma$ such that $\varphi(\Delta) \cap \varphi'(\Gamma) = \{\emptyset\}$.  We will view this pseudo-cobordism $\Omega$ as a buffer between $\Delta$ and $\Gamma$ that is a more appropriate analogue of a collar on $\Delta$ or an honest cobordism between $\Delta$ and $\Gamma$. 

We begin by connecting Pachner's result that a combinatorial sphere bounds a shellable ball (Theorem \ref{shellable-ball}) to our definition of shellable pseudo-cobordisms.   Suppose a simplicial complex $\mathcal{K}$ is a shellable ball with shelling order $F_1, \ldots, F_t$, and let $\mathcal{K}_i$ be the simplicial complex induced by the first $i$ facets in the shelling order.  It follows from a result of Danaraj and Klee \cite{Danaraj-Klee} 
that each of these $\mathcal{K}_i$ complexes is a ball.  Moreover, the boundary of each $\mathcal{K}_{i+1}$ is obtained from the boundary of $\mathcal{K}_i$ by a bistellar flip.  Therefore, this shelling sequence encodes the sequence of bistellar flips transforming the boundary of a simplex (i.e., the boundary of $\overline{F}_1$) to the boundary of $\mathcal{K}$. In fact, it follows from the next lemma that the simplicial complex $\mathcal{K} \setminus F_1$ ($\mathcal{K}$ with the facet $F_1$ removed, but its boundary still intact) is a shellable pseudo-cobordism between $\partial\overline{F}_1$ and $\partial\mathcal{K}$. 

\begin{lemma} \label{reverse-shelling-ball}
Let $\mathcal{K}$ be a simplicial complex. Assume further that $\mathcal{K}$ is a shellable ball with shelling order $F_1, F_2, \ldots, F_t$. Then $(\mathcal{K}, \partial\mathcal{K})$ is a shellable relative complex with shelling order $F_t, F_{t-1},\ldots, F_1$. 
\end{lemma}

\begin{proof}
Let $B_j$ denote the restriction face for the $j^{\text{th}}$ facet in the given shelling order of $\mathcal{K}$, and $A_j$ the complementary face:
\[
[\emptyset, F_j] \setminus \bigcup_{i=1}^{j-1} [\emptyset, F_i] = [B_j, F_j] \quad \text{and} \quad A_j := F_j \setminus B_j.
\]
We claim that in the proposed shelling order on $(\mathcal{K}, \partial\mathcal{K})$, 
\begin{equation}\label{eqn:reverse-shelling}
[\emptyset, F_j] \setminus \Bigg( \bigcup_{i=j+1}^t [\emptyset, F_i] \cup \partial \mathcal{K} \Bigg) = [A_j,F_j].
\end{equation}

We prove the claim by induction on the number of facets of $\mathcal{K}$.  The result indeed holds for $t=1$, in which case $\mathcal{K} = \overline{F_1}$ is a simplex.

Now suppose $t>1$.  First, we show that Eq.~\eqref{eqn:reverse-shelling} holds when $j=t$.  Note that a face $\tau \subseteq F_t$ belongs to $\partial\mathcal{K}$ if and only if $\tau$ is contained in some facet of $\partial \overline{F}_t$ that contains $B_t$.  The subcomplex of $\partial \overline{F}_t$ induced by the faces containing $B_t$ is $\partial \overline{A}_t * \overline{B}_t$, and so $\tau \in \partial\overline{F}_t \cap \partial \mathcal{K}$ if and only if $A_t \not\subseteq \tau$.  Thus $[\emptyset, F_t] \setminus \partial \mathcal{K} = [A_t,F_t]$.

Next, let $\mathcal{K}'$ be the subcomplex of $\mathcal{K}$ induced by the facets $F_1, \ldots, F_{t-1}$. This $\mathcal{K}'$ is also a shellable ball.  Further, $\partial \mathcal{K}' = \left(\partial \mathcal{K} \setminus (\partial \overline{A}_t*\overline{B}_t)\right) \cup (\overline{A}_t * \partial\overline{B}_t)$.  For $j < t$, observe that
$$
[\emptyset, F_j] \setminus \Bigg( \bigcup_{i=j+1}^t [\emptyset, F_i] \cup \partial \mathcal{K} \Bigg)
=
[\emptyset, F_j] \setminus \Bigg( \bigcup_{i=j+1}^{t-1} [\emptyset, F_i] \cup \partial \mathcal{K}' \Bigg).
$$
This is because $$\bigcup_{i=j+1}^t [\emptyset, F_i] \cup \partial \mathcal{K} = \bigcup_{i=j+1}^{t-1} [\emptyset, F_i] \cup \partial \mathcal{K}' \sqcup [B_t,F_t],$$ but no faces in $[\emptyset, F_j]$ belong to $[B_t,F_t]$.  Then Eq.~\eqref{eqn:reverse-shelling} holds for index $j$ since the inductive hypothesis applies to $\mathcal{K}'$ and hence $[\emptyset, F_j] \setminus \left( \bigcup_{i=j+1}^{t-1} [\emptyset, F_i] \cup \partial \mathcal{K}' \right) = [A_j,F_j]$.
\end{proof}

This leads to the main result of this subsection, which allows us to construct a shellable pseudo-cobordism $\Omega$ between two PL homeomorphic closed combinatorial manifolds, say $\Delta$ and $\Gamma$, such that the images of $\Delta$ and $\Gamma$ in $\Omega$ are disjoint. 

\begin{lemma} \label{lem:eliminate-face}
Let $\Delta$ be a closed combinatorial $d$-manifold and let $\tau \in \Delta$.  There exists a closed combinatorial $d$-manifold $\Delta'$ and a $(d+1)$-dimensional nonpure pseudomanifold $\Omega$ such that:
\begin{enumerate}
\item $\Delta'$ is PL homeomorphic to $\Delta$,
\item $\Omega$ is a shellable pseudo-cobordism between $\Delta$ and $\Delta'$, and
\item $\tau \notin \Delta'$.
\end{enumerate}
\end{lemma}

\begin{proof}
By Theorem \ref{shellable-ball}, there is a shellable ball, $\mathcal{K}$, such that $\partial\mathcal{K} = \lk_{\Delta}(\tau)$ and such that $\partial \mathcal{K}$ is an induced subcomplex of $\mathcal{K}$.  Let $F_1, \ldots, F_t$ be a shelling order on $\mathcal{K}$.  As in the proof of Lemma \ref{reverse-shelling-ball}, let $[\emptyset, F_j] \setminus \bigcup_{i<j} [\emptyset, F_i] = [B_j,F_j]$ and $A_j:= F_j \setminus B_j$.  For each $j$, let $G_j:= F_j \cup \tau$.

Consider $\Omega:= \Delta \cup ([\emptyset,\tau]*\mathcal{K})$ and $\Delta':= (\Delta \setminus \tau) \cup ([\emptyset, \tau)*\mathcal{K})$.  Since $\partial \mathcal{K}$ is an induced subcomplex of $\mathcal{K}$, $\Delta'$ is a simplicial complex.  Furthermore, $\Omega$ is a nonpure pseudomanifold whose $(d+1)$-faces are $G_1, \ldots, G_t$.  We claim that the following holds (note that the order is perhaps different from one's initial intuition):
\begin{enumerate}
\item $G_t, \ldots, G_1$ is a shelling order on $(\Omega, \Delta)$ and
\item $G_1, \ldots, G_t$ is a shelling order on $(\Omega, \Delta')$.
\end{enumerate}
This claim will establish that $\Omega$ is a shellable pseudo-cobordism between $\Delta$ and $\Delta'$, which in turn would imply that  $\Delta'$ is bistellar equivalent, hence PL homeomorphic, to $\Delta$.  The result follows  since  $\tau \notin \Delta'$ by definition of $\Delta'$.

To establish the claim, we first note that $G_t, \ldots, G_1$ is a shelling order on $(\Omega,\Delta)$. This is because
\[
\sigma \in [\emptyset, G_j] \setminus \Bigg( \bigcup_{i=j+1}^t [\emptyset, G_i] \cup \Delta \Bigg) \;\;
\Leftrightarrow \;\;
\sigma \setminus \tau \in [\emptyset, F_j] \setminus \Bigg( \bigcup_{i=j+1}^t [\emptyset, F_i] \cup \partial \mathcal{K} \Bigg)
=[A_j,F_j],
\]
where the last equality holds by Eq.~\eqref{eqn:reverse-shelling}. Thus, $A_j \subseteq \sigma \setminus \tau$, that is, $\sigma \in [A_j, G_j]$.

Next, we show that $G_1, \ldots, G_t$ is a shelling order on $(\Omega, \Delta')$.  Observe that 
$$
\sigma \in [\emptyset, G_j] \setminus \Bigg( \bigcup_{i=1}^{j-1} [\emptyset, G_i] \cup \Delta' \Bigg)
\text{\quad if and only if \quad} 
\tau \subseteq \sigma \text{ and }\sigma \setminus \tau \in [\emptyset, F_j] \setminus \bigcup_{i=1}^{j-1} [\emptyset, F_i].
$$  
Indeed, if $\sigma \subseteq G_j$ and $\sigma \cap \tau \subsetneq \tau$, then $\sigma \in [\emptyset, \tau)*\mathcal{K} \subseteq \Delta'$.  Otherwise, if $\tau \subseteq \sigma$, then $\sigma \notin \Delta'$ and, for each $i$, $\sigma \subseteq G_i$ if and only if $\sigma \setminus \tau \subseteq F_i$.  Thus 
$$
[\emptyset, G_j] \setminus \Bigg( \bigcup_{i=1}^{j-1} [\emptyset, G_i] \cup \Delta' \Bigg) = [B_j \cup \tau, G_j],
$$
and the statement follows.
\end{proof}

\begin{corollary} \label{step2}
Let $\Delta$ and $\Gamma$ be PL homeomorphic closed combinatorial $d$-manifolds.  There exists a shellable pseudo-cobordism $(\Omega, \varphi, \psi)$ between $\Delta$ and $\Gamma$ such that $\varphi(\Delta) \cap \psi(\Gamma) = \{\emptyset\}$.
\end{corollary}

\begin{proof}
We list the vertices of $\Delta$ in an arbitrary order, $v_1, \ldots, v_n$, and define a sequence of closed combinatorial manifolds $\Delta_0, \Delta_1, \ldots, \Delta_{n}$ in the following way. Let $\Delta_0 := \Delta$. For $1 \leq j \leq n$ we apply Lemma \ref{lem:eliminate-face} to the face $\tau = \{v_j\} \in \Delta_{j-1}$ to obtain a new closed combinatorial manifold $\Delta_j$ and a shellable pseudo-cobordism $\Omega_j$ between $\Delta_{j-1}$ and $\Delta_j$. Since $\Delta_j = (\Delta_{j-1} \setminus v_j) \cup \mathcal{K}_j$, it follows that $\{v_j\} \notin \Delta_j$ and $\{v_{j+1}\} \in \Delta_j$.

Let $\Delta' = \Delta_n$, and let $\widetilde{\Omega}$ be the composition of the shellable pseudo-cobordisms $\Omega_1, \Omega_2, \ldots, \Omega_n$.  By Lemma \ref{lem:GluingCobord}, $(\widetilde{\Omega}, \id, \id)$ is a shellable pseudo-cobordism between $\Delta$ and $\Delta'$. In particular, $\Delta'$ is PL homeomorphic to $\Delta$, and hence PL homeomorphic to $\Gamma$. By Theorem \ref{thm:BistCobord}, there is a shellable pseudo-cobordism $(\widetilde{\Omega}', \id, \psi)$ between $\Delta'$ and $\Gamma$. The composition $\Omega:= \widetilde{\Omega} \cup_{\Delta'} \widetilde{\Omega}'$ is a shellable pseudo-cobordism between $\Delta$ and $\Gamma$.

Since no vertex of $\Delta$ belongs to $\Delta'$, these two complexes share no nonempty faces: $\Delta \cap \Delta' = \{\emptyset\}$. As a consequence, $\Delta \cap \psi(\Gamma) = \{\emptyset\}$.
\end{proof}

\begin{remark}
\label{rem:ConeLink}
If the links of all vertices are shellable, we can take $\mathcal{K}_j = \overline{v'_j} \ast \lk_{\Delta_{j-1}} (\{v_j\})$. In this case each of $\Delta_j$, and in particular $\Delta'$, is isomorphic to $\Delta$.
\end{remark}

\begin{example}
Let $\Delta$ be the boundary of a triangle. Applying the procedure from the first half of the proof of Corollary \ref{step2}, we obtain a cobordism of $\Delta$ with $\Delta'$ shown in Figure \ref{fig:DeltaPrime}. Here $\mathcal{K}_j$ is the cone over the link, as in Remark \ref{rem:ConeLink}.

\begin{figure}[ht]
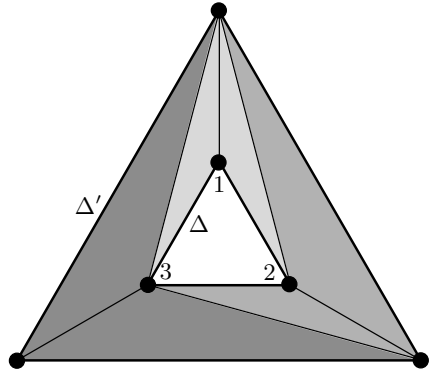
\caption{A shellable pseudo-cobordism with disjoint ends.}
\label{fig:DeltaPrime}
\end{figure}
\end{example}

\subsection{Stellar subdivisions preserve shellability}

Provan and Billera \cite{Provan-Billera-80} showed that shellability of simplicial complexes is preserved under stellar subdivision.  Our next goal is to prove an analogous result for shellable pseudo-cobordisms; i.e., if $\Omega$ is a shellable pseudo-cobordism between $\Delta$ and $\Delta'$, and $\sigma \in \Omega$, then $\sd_{\sigma}(\Omega)$ is a shellable pseudo-cobordism between $\sd_{\sigma}(\Delta)$ and $\sd_{\sigma}(\Delta')$.  In proving this result, we make use of the following elementary facts. 

\begin{lemma} \ 
\begin{enumerate}
\item Any ordering of the facets of the boundary of a simplex is a shelling order.
\item If $F = A \sqcup B$, then $\partial \overline{F} = (\partial\overline{A}*\overline{B}) \cup (\overline{A}*\partial\overline{B})$. 
\item Suppose $\Gamma$ and $\Delta$ are pure simplicial complexes with disjoint vertex sets.  
\begin{enumerate}
\item If $\Gamma$ and $\Delta$ are shellable,  then $\Gamma * \Delta$ is also shellable.  Further, if $F_1, \ldots, F_s$ is a shelling order for $\Gamma$ and $G_1, \ldots, G_t$ is a shelling order for $\Delta$, then $F_1 \cup G_1, F_1 \cup G_2, \ldots, F_1 \cup G_t, F_2 \cup G_1, \ldots, F_s \cup G_t$ is a shelling order for $\Gamma*\Delta$.
\item If $F \in \Gamma$ and $G \in \Delta$, then $\lk_{\Gamma*\Delta}(F \cup G) = \lk_{\Gamma}(F) * \lk_{\Delta}(G)$.
\item If $F \in \Gamma$ and $G \in \Delta$, then $(\Gamma*\Delta) \setminus (F \cup G) = \left((\Gamma \setminus F)*\Delta\right) \cup \left(\Gamma*(\Delta \setminus G)\right)$.
\end{enumerate} 
\end{enumerate}
\end{lemma}

\begin{proposition} \label{subdivision-preserves-shellability}
Let $\Delta$ and $\Delta'$ be closed combinatorial $d$-manifolds and let $\Omega$ be a shellable pseudo-cobordism between $\Delta$ and $\Delta'$. Identify $\Delta$ and $\Delta'$ with their isomorphic images in $\Omega$. Then for any face $\sigma \in \Omega$, the subdivision $\sd_{\sigma}(\Omega)$ is a shellable pseudo-cobordism between $\sd_{\sigma}(\Delta)$ and $\sd_{\sigma}(\Delta')$ (with the convention that $\sd_{\tau}(\Gamma) = \Gamma$ if $\tau \notin \Gamma)$.
\end{proposition}
\begin{proof}
We begin the proof with several reductions.  

First, it suffices to consider the case that $\Omega$ is an elementary pseudo-cobordism.  Indeed,
by Proposition \ref{prp:ShellCompos}, every shellable pseudo-cobordism can be represented as a composition of elementary ones.
A stellar subdivision of a composition of pseudo-cobordisms is the composition of their stellar subdivisions. Thus, if we show that subdivisions of elementary pseudo-cobordisms are shellable, then the result would follow from Lemma \ref{lem:GluingCobord} on the composition of shellable pseudo-cobordisms.

Let $\Omega$ be an elementary pseudo-cobordism. Then $\Omega$ is a simplicial complex of the form
\[
\Omega = \Delta \cup (\overline{A}*\overline{B}) \quad \text{with} \quad \Delta = \Omega \setminus B \quad \text{and} \quad \Delta' = \Omega \setminus A.
\]
Now, if $\sigma \not\subseteq A \cup B$, then $\sigma \in \Omega \setminus (\overline{A}*\overline{B}) = \Delta \cap \Delta'$, in which case $\sd_\sigma(\Omega)=\Omega$ is indeed an elementary pseudo-cobordism between the combinatorial manifolds $\sd_\sigma(\Delta)$ and $\sd_\sigma(\Delta')$.

Hence we may assume that $\sigma \subseteq A \cup B$. As the parts of $\Delta$ and $\Delta'$ lying outside of $\overline{A}*\overline{B}$ play no role in the subdivision and in the shelling, it suffices to consider the case that
\[
\Omega = \overline{A}*\overline{B}, \quad \Delta = \overline{A}*\partial\overline{B}, \quad \Delta' = \partial\overline{A}*\overline{B}.
\]
It is clear that for any $\sigma \subseteq A \cup B$ the subdivided simplex $\sd_{\sigma}(\Omega)$ is a pseudo-cobordism between $\sd_{\sigma}(\Delta)$ and $\sd_{\sigma}(\Delta')$.  Therefore, we need only to exhibit that this pseudo-cobordism is shellable according to Definition \ref{def:shellablePseudoCobord}.

In order to complete the proof, we must consider three cases based on the relative positions of the faces $A$, $B$, and $\sigma$.

\smallskip\noindent\textbf{Case 1:} $\sigma \subseteq A$ (or by symmetry, $\sigma \subseteq B$).

First, suppose $\sigma \subsetneq A$ and let $X = A \setminus \sigma$ (which is nonempty).  Observe that 
\begin{eqnarray*}
\sd_{\sigma}(\Omega) &=& \overline{a}*\partial\overline{\sigma}*\overline{X}*\overline{B}, \\
\sd_{\sigma}(\Delta) &=& \overline{a}*\partial\overline{\sigma}*\overline{X}*\partial\overline{B}, \text{ \qquad and} \\
\sd_{\sigma}(\Delta') &=& \partial\overline{\sigma} * \overline{X}*\overline{B} \cup \overline{a}*\partial\overline{\sigma}*\partial\overline{X}*\overline{B}.
\end{eqnarray*}

Hence $$(\sd_{\sigma}(\Omega), \sd_{\sigma}(\Delta)) = \{B \cup F \mid F \in \overline{a}*\partial\overline{\sigma} *\overline{X}\},$$  and so as a set system, $(\sd_{\sigma}(\Omega), \sd_{\sigma}(\Delta))$ is isomorphic to the face poset of the simplicial complex $\overline{a}*\partial\overline{\sigma}*\overline{X}$, which is shellable under \textit{any} ordering of its facets.  Thus $(\sd_{\sigma}(\Omega), \sd_{\sigma}(\Delta))$ is also shellable under any ordering of its facets.  

Similarly, $$(\sd_{\sigma}(\Omega), \sd_{\sigma}(\Delta')) = \{a \cup X \cup F \mid F \in \partial\overline{\sigma}*\overline{B}\},$$ is shellable under any ordering of its facets.  Therefore, if we choose an arbitrary ordering $F_1, F_2, \ldots, F_t$ of the $(d+1)$-faces of $\sd_{\sigma}(\Omega)$, then $F_1, F_2, \ldots, F_t$ is a shelling order on $(\Omega, \Delta)$ and $F_t, F_{t-1},\ldots,F_1$ is a shelling order on $(\Omega,\Delta')$.  Thus $\sd_{\sigma}(\Omega)$ is a shellable pseudo-cobordism between $\Delta$ and $\Delta'$.

On the other hand, if $\sigma = A$, then 
\begin{eqnarray*}
\sd_{\sigma}(\Omega) &=& \overline{a}*\partial\overline{A}*\overline{B}, \\
\sd_{\sigma}(\Delta) &=& \overline{a}*\partial\overline{A}*\partial\overline{B}, \text{\qquad and} \\
\sd_{\sigma}(\Delta') &=& \partial \overline{A}*\overline{B}.
\end{eqnarray*}
Thus $(\sd_{\sigma}(\Omega),\sd_{\sigma}(\Delta)) = \{B \cup F \mid F \in \overline{a}*\partial\overline{A}\}$ and $(\sd_{\sigma}(\Omega), \sd_{\sigma}(\Delta')) = \{a \cup F \mid F \in \partial \overline{A}*\overline{B}\}$.  As above, each of these relative complexes is shellable under any ordering of its facets, and hence $\sd_{\sigma}(\Omega)$ is a shellable pseudo-cobordism between $\sd_{\sigma}(\Delta)$ and $\sd_{\sigma}(\Delta')$.

\smallskip\noindent\textbf{Case 2:} $A \subsetneq \sigma$ (or by symmetry, $B \subsetneq \sigma$).

First we consider the case that $\sigma = A \cup B$ so that 
\[\sd_{\sigma}(\Omega) = \overline{a}*\partial\overline{\sigma} = \overline{a}*(\partial\overline{A}*\overline{B} \cup \overline{A}*\partial\overline{B}),  \;\; \sd_{\sigma}(\Delta) = \Delta, \; \mbox{ and } \; \sd_{\sigma}(\Delta') = \Delta'.
\]
Let $\Omega':= \overline{a}*\partial\overline{A}*\overline{B}$ and $\Omega'':= \overline{a}*\overline{A}*\partial\overline{B}$.  Note that 
\begin{eqnarray*}
\Omega' &=& \sd_A(\Omega), \\
\Omega'' &=& \sd_B(\Omega), \\
\Omega' \cup \Omega'' &=& \Omega, \qquad\mbox{and} \\
\Omega' \cap \Omega'' &=& \overline{a}*\partial\overline{A} *\partial\overline{B} = \sd_{A}(\Delta) = \sd_{B}(\Delta').
\end{eqnarray*}

By our observations in Case 1, $\Omega''$ is a shellable pseudo-cobordism between $\sd_B(\Delta) = \Delta = \sd_{\sigma}(\Delta)$ and $\sd_B(\Delta')=\sd_A(\Delta)$, while $\Omega'$ is a shellable pseudo-cobordism between $\sd_A(\Delta)$ and $\sd_A(\Delta') = \Delta' = \sd_{\sigma}(\Delta')$.  Therefore, by Lemma \ref{lem:GluingCobord}, the composition $\Omega = \Omega' \cup \Omega''$ is a shellable pseudo-cobordism between $\Delta = \sd_{\sigma}(\Delta)$ and $\Delta' = \sd_{\sigma}(\Delta')$.

Next, we consider the case that $\sigma \neq A \cup B$.  Let $X = \sigma \setminus A = \sigma \cap B$ and $Y = B \setminus \sigma$ so that $B = X \cup Y$, $\sigma = A \cup X$, and $X$ and $Y$ are nonempty.

Observe that
\begin{eqnarray*}
\sd_{\sigma}(\Omega) &=& \overline{a}*\partial\overline{\sigma} * \overline{Y} \\
&=& \overline{a}*\overline{A}*\partial\overline{X}*\overline{Y} \ \cup \  \overline{a}*\partial\overline{A}*\overline{X}*\overline{Y} \\
&=& \overline{a}*\overline{A}*\partial\overline{X}*\overline{Y} \ \sqcup \ \{X \cup F \mid F \in \overline{a}*\partial\overline{A}*\overline{Y}\} \text{\qquad and} \\
\sd_{\sigma}(\Delta) &=& \Delta \setminus \sigma \ \cup \  \overline{a}*\partial\overline{\sigma}*\lk_{\Delta}(\sigma) \\
&=& (\partial\overline{A}*\partial\overline{B} \ \cup \  \overline{A}*\partial\overline{X} *\overline{Y}) \ \cup \  \overline{a}*\partial\overline{\sigma}*\partial\overline{Y} \\
&=& (\partial\overline{A}*\partial\overline{X}*\overline{Y} \ \cup \  \partial\overline{A}*\overline{X}*\partial\overline{Y} \ \cup \  \overline{A}*\partial\overline{X}*\overline{Y}) \\
&& \cup\ \overline{a}*\overline{A}*\partial\overline{X}*\partial\overline{Y} \ \cup \  \overline{a}*\partial\overline{A}*\overline{X}*\partial\overline{Y} \\
&=& \overline{A}*\partial\overline{X}*\overline{Y} \ \cup \  \overline{a}*\overline{A}*\partial\overline{X}*\partial\overline{Y} \ \cup \  \overline{a}*\partial\overline{A}*\overline{X}*\partial\overline{Y}.
\end{eqnarray*}

Thus $$(\sd_{\sigma}(\Omega),\sd_{\sigma}(\Delta)) = \{a \cup Y \cup F \mid F \in \overline{A}*\partial\overline{X}\} \sqcup \{X \cup Y \cup F \mid F \in \overline{a}*\partial\overline{A}\}.$$  Notice that because these two sets are disjoint, it is still the case that any ordering of the facets of $(\sd_{\sigma}(\Omega), \sd_{\sigma}(\Delta))$ is a shelling order.

Finally, since $\sigma \notin \Delta'$, $$\sd_{\sigma}(\Delta') = \partial\overline{A}*\overline{B} = \partial \overline{A}*\overline{X}*\overline{Y}.$$ Writing $$\sd_{\sigma}(\Omega) = \overline{a}*\partial\overline{A}*\overline{X}*\overline{Y} \sqcup \{A \cup F \mid F \in \overline{a}*\partial\overline{X}*\overline{Y}\},$$ we see that 

$$(\sd_{\sigma}(\Omega), \sd_{\sigma}(\Delta')) = \{a \cup F \mid F \in \partial\overline{A}*\overline{X}*\overline{Y}\} \sqcup \{A \cup F \mid F \in \overline{a}*\partial\overline{X}*\overline{Y}\},$$ which is shellable under any ordering of its facets.

\smallskip\noindent\textbf{Case 3:} $\sigma \cap A$, $\sigma \cap B$, $A \setminus \sigma$, and $B \setminus \sigma$ are all nonempty.

As in the previous case, consider $W = A \setminus \sigma$, $X = \sigma \cap A$, $Y = \sigma \cap B$, and $Z = B \setminus \sigma$, all of which are nonempty.  Note $\sigma = X \cup Y$.

As in Case 2, we can write
\begin{eqnarray*}
\sd_{\sigma}(\Omega) &=& \overline{a}*\overline{X}*\partial\overline{Y}*\overline{W}*\overline{Z} \cup \overline{a}*\partial\overline{X}*\overline{Y}*\overline{W}*\overline{Z} \\
&=& \overline{a}*\overline{X}*\partial\overline{Y}*\overline{W}*\overline{Z} \sqcup \{Y \cup F \mid F \in \overline{a}*\partial\overline{X}*\overline{W}*\overline{Z}\} \text{\qquad and} \\
\sd_{\sigma}(\Delta) &=& \overline{a}*\partial\overline{X}*\overline{Y}*\overline{W}*\partial\overline{Z} \cup \overline{a}*\overline{X}*\partial\overline{Y}*\overline{W}*\partial\overline{Z} \cup \overline{X}*\partial\overline{Y}*\overline{W}*\overline{Z}.
\end{eqnarray*}

Thus $$(\sd_{\sigma}(\Omega),\sd_{\sigma}(\Delta)) = \{a \cup Z \cup F \mid F \in \overline{X}*\partial\overline{Y}*\overline{W}\} \sqcup \{Y \cup Z \cup F \mid F \in \overline{a}*\partial\overline{X}*\overline{W}\},$$ which is shellable under any ordering of its facets.  In this case, the roles of $A$ and $B$ are symmetric, and hence $(\sd_{\sigma}(\Omega), \sd_{\sigma}(\Delta'))$ is also shellable under any ordering of its facets.
\end{proof}

As discussed in Section \ref{Section:proof-of-main-theorem}, this completes the proofs of Theorem \ref{color-thm} and Theorem \ref{main-thm}.

\section{Open questions} We close the paper with several remarks and open questions.
\subsection{Moves for balanced manifolds with boundary}
There is an analog of Pachner's theorem for manifolds with boundary: any two PL homeomorphic combinatorial manifolds with boundary are related by a sequence of elementary shellings and inverse shellings, \cite[Theorem 5.10]{Lickorish-99}. Clearly, not every elementary shelling on a balanced simplicial complex preserves balancedness, as it can create edges whose endpoints have the same color. On the other hand, inverse shellings trivially preserve balancedness.

\begin{problem}
Can any two PL homeomorphic balanced combinatorial manifolds with boundary be related by a sequence of elementary inverse shellings and balanced elementary shellings? If not, find another set of local moves that preserve balancedness and allows to relate any two such manifolds.
\end{problem}

\subsection{Restricting the set of cross-flips}
Remark \ref{rem:3Suffice} shows that, in order to transform one balanced triangulation of a combinatorial manifold into another, some restricted set of combinatorially distinct cross-flips does suffice. Namely, one can assume that the ball $D \subset \C_d$ is a subcomplex of $\partial \sigma^{d+1}$ subdivided by the diamond construction from Section \ref{sec:Diamond}. We refer to these cross-flips as \emph{basic cross-flips}.
\begin{problem}
Give an explicit description of basic cross-flips. How many combinatorially distinct basic cross-flips are there?
\end{problem}

In dimension $2$ there are $3$ non-trivial basic cross-flips (or $6$, if the inverses are counted separately), see Remark \ref{rem:3Suffice}.

\begin{remark} In the original version of this paper, we asked if the pentagon move (the second move in the second line of Figure \ref{fig:CrossFlips2}) is necessary. As was very recently shown by Murai and Suzuki \cite{Murai-Suzuki}, the answer is yes; furthermore, if $\Delta$ and $\Gamma$ are two balanced $2$-spheres and neither of the two is the boundary of the cross-polytope, then $\Delta$ can be transformed into $\Gamma$ by using {\em only} pentagon moves!
\end{remark}


\subsection{Even triangulations}
If a triangulation of a closed combinatorial $d$-manifold is balanced, then the link of each of its $(d-2)$-dimensional faces is a polygon with an even number of vertices. The triangulations with the latter property are called \emph{even triangulations}. (Topological properties of even triangulations were recently investigated in \cite{RT-15}.)

The following construction provides an obstruction for an even triangulation to be balanced. Let $\Delta$ be an even triangulation of a closed combinatorial manifold $M$. Pick a facet $\sigma_0$ of $\Delta$. Let $\gamma$ be a closed path in the dual graph of $\Delta$  (also known as the facet-ridge graph) starting and ending in $\sigma_0$. Extending a coloring of $\sigma_0$ along $\gamma$ results in a permutation of vertices of $\sigma_0$; this permutation depends only on the homotopy class of $\gamma$, due to the evenness of $\Delta$. This gives rise to a group homomorphism
\[
\phi_\Delta \colon \pi_1(M, \sigma_0) \to \Sym(\sigma_0)
\]
which we call the \emph{coloring monodromy}. An even triangulation is balanced if and only if its coloring monodromy is trivial (in particular, even triangulations of simply connected closed combinatorial manifolds are balanced). The coloring monodromy was introduced by Steve Fisk \cite{Fisk-77a} and rediscovered in \cite{Joswig-02}. See \cite{Izm-15+} for simple examples and applications.

Cross-flips on even triangulations do not change the coloring monodromy. This leads us to the following question.

\begin{problem}
Let $\Gamma$ and $\Delta$ be two even triangulations of the same combinatorial manifold $M$ with the same coloring monodromy: $\phi_\Gamma = \phi_\Delta$ up to conjugation. Does it follow that there is a sequence of cross-flips transforming $\Gamma$ into $\Delta$?
\end{problem}

More generally, a coloring monodromy is defined for every triangulation $\Delta$ of $M$ as a homomorphism from $\pi_1(M\setminus\Delta_{\odd}, \sigma_0)$ to $\Sym(\sigma_0)$, where  $\Delta_{\odd}$ is the odd subcomplex of $\Delta$, see \cite{Fisk-77a, Izm-15+, Joswig-02}. This allows us to further generalize the above question.


\bibliography{biblio}
\bibliographystyle{abbrv}

\end{document}